%% file: rigorousCME.tex
\documentclass[a4paper, twoside]{scrartcl}
\input{Packages}
\input{Macros}

\usepackage{etex}
\usepackage[a4paper, top=88pt, bottom=88pt, left=72pt, right=72pt, headsep=16pt, footskip=28pt]{geometry}
\usepackage{hyperref}
\hypersetup{
	colorlinks,
	linkcolor=blue,
	citecolor=blue,
	urlcolor=blue,
}
\newcommand*{\email}[1]{\bgroup\color{blue}\href{mailto:#1}{#1}\egroup}
\renewcommand*{\url}[1]{\bgroup\color{blue}\href{#1}{#1}\egroup}
\usepackage{enumitem}
\setlist[enumerate]{nosep}
\setlist[itemize]{nosep}
\usepackage{mleftright} \mleftright
\renewcommand{\qedsymbol}{$\blacksquare$}
\renewenvironment{proof}[1][\proofname]{\noindent{\bfseries #1.} }{\hfill \qedsymbol \medskip}
\usepackage{scrlayer-scrpage, xhfill}
\automark[section]{section}
\setkomafont{pageheadfoot}{\normalcolor\sffamily}
\setkomafont{pagenumber}{\normalfont\normalsize\sffamily}
\clearpairofpagestyles
\let\oldtitle\title
\renewcommand{\title}[1]{\oldtitle{#1}\newcommand{\theshorttitle}{#1}}
\newcommand{\shorttitle}[1]{\renewcommand{\theshorttitle}{#1}}
\let\oldauthor\author
\renewcommand{\author}[1]{\oldauthor{#1}\newcommand{\theshortauthor}{#1}}
\newcommand{\shortauthor}[1]{\renewcommand{\theshortauthor}{#1}}
\cohead{\xrfill[0.525ex]{0.6pt}~\theshorttitle~\xrfill[0.525ex]{0.6pt}}
\cehead{\xrfill[0.525ex]{0.6pt}~\theshortauthor~\xrfill[0.525ex]{0.6pt}}
\newcommand*{\affilref}[1]{\ref{affiliation#1}}
\newcommand*{\affiliation}[3]{
	\footnotetext[#1]{\label{affiliation#2} #3}
}

\title{A Rigorous Theory of\\Conditional Mean Embeddings}
\shorttitle{A Rigorous Theory of Conditional Mean Embeddings}
\author{%
	Ilja~Klebanov\textsuperscript{\affilref{ZIB}}%
	\and
	Ingmar~Schuster\textsuperscript{\affilref{Zalando}}%
	\and
	T.~J.~Sullivan\textsuperscript{\affilref{ZIB},\affilref{FUB},\affilref{Warwick}}%
}
\shortauthor{I.~Klebanov, I.~Schuster, and T.~J.~Sullivan}

\begin{document}
\maketitle
\affiliation{1}{ZIB}{Zuse Institute Berlin, Takustra{\ss}e 7, 14195 Berlin, Germany. \email{klebanov@zib.de}, \email{sullivan@zib.de}}
\affiliation{2}{Zalando}{Zalando SE, 11501 Berlin, Germany. \email{ingmar.schuster@zalando.de}}
\affiliation{3}{FUB}{Freie Universit{\"a}t Berlin, Arnimallee 6, 14195 Berlin, Germany. \email{t.j.sullivan@fu-berlin.de}}
\affiliation{4}{Warwick}{Mathematics Institute and School of Engineering, The University of Warwick, Coventry, CV4 7AL, United Kingdom. \email{t.j.sullivan@warwick.ac.uk}}

\begin{abstract}
	\noindent\textbf{\textsf{Abstract:}}
	\input{./chunk-abstract.tex}
	
	\medskip

	\noindent\textbf{\textsf{Keywords:}}
	\input{./chunk-keywords.tex}

	\medskip

	\noindent\textbf{\textsf{2010 Mathematics Subject Classification:}}
	\input{./chunk-msc.tex}
\end{abstract}

\input{section-01-introduction}
\input{section-02-setup}
\input{section-03-assumptions}
\input{section-04-centred}

\input{section-05-uncentred}
\input{section-06-Gaussian_conditioning}
\input{section-07-connection_Gaussian_CME}
\input{section-08-closing}

\section*{Acknowledgements}
\addcontentsline{toc}{section}{Acknowledgements}

\input{./chunk-acknowledgements.tex}

\appendix

\input{appendix-a-technical}
\input{appendix-c-empirical}

\bibliographystyle{abbrvnat}
\bibliography{myBibliography}
\addcontentsline{toc}{section}{References}

\end{document}

%% file: Packages.tex
\usepackage{amsfonts,amssymb,amsmath,amsthm,amstext,amssymb,mathtools}
\usepackage{caption,subcaption,float,color,xcolor,graphicx,enumerate,hyperref}
\usepackage{dsfont}  
\usepackage[normalem]{ulem} 
\usepackage[authoryear,sort,round]{natbib}  
\usepackage[nameinlink]{cleveref}

\crefname{counterexample}{Counterexample}{Counterexamples}
\Crefname{counterexample}{Counterexample}{Counterexamples}

\usepackage{adjustbox}
\usepackage{tikz}
\usetikzlibrary{bayesnet,er,trees,shapes.symbols,mindmap,arrows,decorations,shapes.misc,shapes.arrows,chains,matrix,positioning,scopes,decorations.pathmorphing,patterns}
\usepackage{tikz-cd} 

\usetikzlibrary{decorations.markings}
\tikzset{
  negate/.style={
    decoration={
      markings,
      mark= at position 0.5 with {
        \node[transform shape] (tempnode) {$/$};
      },
    },
    postaction={decorate},
  },
}

\usepackage{dashbox}
\usepackage{chngcntr}
\usepackage{thmtools,thm-restate}
\usepackage{array}
\newcolumntype{x}[1]{>{\centering\arraybackslash\hspace{0pt}}p{#1}}
\usepackage{booktabs,colortbl,xcolor,xfrac}
\usepackage{multirow}

%% file: Macros.tex
\newcommand*{\cX}{\mathcal{X}}
\newcommand*{\cY}{\mathcal{Y}}
\newcommand*{\cH}{\mathcal{H}}
\newcommand*{\cG}{\mathcal{G}}
\newcommand*{\cF}{\mathcal{F}}

\newcommand*{\cP}{\mathcal{P}}

\newcommand*{\cL}{\mathcal{L}}
\newcommand*{\cC}{\mathcal{C}}

\newcommand*{\fm}{\mathfrak{m}}
\newcommand*{\ff}{\mathfrak{f}}

\newcommand*{\uu}{\prescript{u}{} }

\newcommand*{\bR}{\mathbb{R}}
\newcommand*{\bN}{\mathbb{N}}

\newcommand*{\bE}{\mathbb{E}}
\newcommand*{\bV}{\mathbb{V}}
\newcommand*{\bP}{\mathbb{P}}
\newcommand*{\bQ}{\mathbb{Q}}
\newcommand*{\cN}{\mathcal{N}}

\newcommand*{\Cov}{\mathbb{C}\mathrm{ov}}

\newcommand*{\quark}{\setbox0\hbox{$x$}\hbox to\wd0{\hss$\cdot$\hss}}
\newcommand*{\rd}{\mathrm{d}}

\DeclareMathOperator{\spn}{span}
\DeclareMathOperator{\supp}{supp}
\DeclareMathOperator{\Id}{Id}

\DeclareMathOperator*{\argmin}{arg\,min}
\DeclareMathOperator{\ran}{ran}
\DeclareMathOperator{\dom}{dom}

\definecolor{darkgreen}{rgb}{0,0.4,0}

\newcommand*{\defeq}{\coloneqq}

\theoremstyle{plain}
\newtheorem{theorem}{Theorem}[section]
\newtheorem{proposition}[theorem]{Proposition}
\newtheorem{lemma}[theorem]{Lemma}
\newtheorem{corollary}[theorem]{Corollary}
\theoremstyle{definition}
\newtheorem{definition}[theorem]{Definition}

\newtheorem{example}[theorem]{Example}

\newtheorem{remark}[theorem]{Remark}

\newtheorem{assumption}[theorem]{Assumption}

\newcommand{\absval}[1]{\vert #1 \vert}
\newcommand{\innerprod}[2]{\langle #1 , #2 \rangle}
\newcommand{\norm}[1]{\Vert #1 \Vert}

\newcommand{\biginnerprod}[2]{\bigl\langle #1 , #2 \bigr\rangle}
\newcommand{\bignorm}[1]{\bigl\Vert #1 \bigr\Vert}

\newcommand{\Innerprod}[2]{\left\langle #1 , #2 \right\rangle}

\numberwithin{equation}{section}
\numberwithin{figure}{section}
\numberwithin{table}{section}

\newtheorem{assumpA}{Assumption}

\newtheorem{assumpB}{Assumption}

\newtheorem{assumpC}{Assumption}

\newtheorem{assumpAprime}{Assumption}

\newtheorem{assumpBprime}{Assumption}

\newtheorem{assumpCdoubleprime}{Assumption}

\definecolor{verylightgray}{rgb}{0.85,0.85,0.85}

\newcommand\mlnode[1]{\fbox{\begin{tabular}{@{}c@{}}#1\end{tabular}}} 
\newcommand\dashedmlnode[1]{\dbox{\begin{tabular}{@{}c@{}}#1\end{tabular}}}
\newcommand\filledmlnode[1]{\fcolorbox{black}{verylightgray}{\begin{tabular}{@{}c@{}}#1\end{tabular}}}
\definecolor{boxgreen}{rgb}{0,0.7,0}
\newcommand\greenmlnode[1]{\fcolorbox{boxgreen}{white}{\begin{tabular}{@{}c@{}}#1\end{tabular}}}
\definecolor{boxred}{rgb}{0.7,0,0}

%% file: chunk-abstract.tex
Conditional mean embeddings (CMEs) have proven themselves to be a powerful tool in many machine learning applications.
They allow the efficient conditioning of probability distributions within the corresponding reproducing kernel Hilbert spaces (RKHSs) by providing a linear-algebraic relation for the kernel mean embeddings of the respective joint and conditional probability distributions.
Both centred and uncentred covariance operators have been used to define CMEs in the existing literature.
In this paper, we develop a mathematically rigorous theory for both variants, discuss the merits and problems of each, and significantly weaken the conditions for applicability of CMEs.
In the course of this, we demonstrate a beautiful connection to Gaussian conditioning in Hilbert spaces.

%% file: chunk-keywords.tex
conditional mean embedding, kernel mean embedding, Gaussian measure, reproducing kernel Hilbert space.

%% file: chunk-msc.tex
46E22, 62J02, 28C20.

%% file: section-01-introduction.tex

\section{Introduction}
\label{section:Intro}

Reproducing kernel Hilbert spaces (RKHSs) have long been popular tools in machine learning because of the powerful property --- often called the ``kernel trick'' --- that many problems posed in terms of the base set $\cX$ of the RKHS $\cH$ (e.g.\ classification into two or more classes) become linear-algebraic problems in $\cH$ under the embedding of $\cX$ into $\cH$ induced by the reproducing kernel $k \colon \cX \times \cX \to \bR$.
This insight has been used to define the \emph{kernel mean embedding} (KME; \citealp{berlinet2004rkhs}; \citealp{smola2007embedding}) $\mu_{X} \in \cH$ of an $\cX$-valued random variable $X$ as the $\cH$-valued mean of the embedded random variable $k(X, \quark)$, and also the \emph{conditional mean embedding} (CME; \citealp{fukumizu2004dimensionality}, \citealp{song2009hilbert}), which seeks to perform conditioning of the original random variable $X$ through application of the Gaussian conditioning formula (also known as the K\'alm\'an update) to the embedded \emph{non-Gaussian} random variable $k(X, \quark)$.
This article aims to provide rigorous mathematical foundations for this attractive but apparently na{\"\i}ve approach to conditional probability, and hence to Bayesian inference.

\begin{figure}[t]
	\centering
	\adjustbox{scale=0.9}{
		\begin{tikzcd}[column sep=4em,row sep=7em]
			\textbf{original spaces } \cX,\cY
			&
			\textbf{RKHS feature spaces } \cH,\cG
			\\[-2.2cm]
			\begin{cases}
				\begin{rcases}
					\qquad\ x\in\cX \\
					\qquad X\sim\bP_{X} \\
					\qquad \, Y\sim\bP_Y \\
					(X,Y)\sim\bP_{XY}
				\end{rcases}
			\end{cases}
			\arrow{r}[sloped,above]{\text{embed}}[sloped,below]{\psi,\varphi}
			\arrow{d}[sloped,above]{\text{conditioning on}}[sloped,below]{X=x}
			&
			\begin{cases}
				\begin{rcases}
					\varphi(x) \\
					\psi(Y),\, \varphi(X) \\
					\mu_{Y},\, C_{Y},\, C_{YX} \\
					\mu_{X},\, C_{XY},\, C_{X} \\
				\end{rcases}
			\end{cases}
			\arrow{d}[sloped,above]{\text{\textcolor{blue}{conditional mean}}}[sloped,below]{\text{\textcolor{blue}{embedding}}}
			\\
			(Y|X=x) \sim \bP_{Y|X=x}
			\arrow{r}{\text{embed}}
			&
			\fbox{$\mu_{Y|X=x} = \begin{cases} C_{YX}C_{X}^{-1}\varphi(x) \phantom{ooooooo} \text{according to } \eqref{equ:CMEmu} \\ \mu_{Y} + (C_{X}^{\dagger} C_{XY})^{\ast} \, (\varphi(x) - \mu_{X}) \phantom{oo} \text{by } \eqref{equ:myCME} \\ (\uu{C}_{X}^{\dagger} \uu{C}_{XY})^{\ast} \, \varphi(x) \phantom{ooooooooooii} \text{by } \eqref{equ:myCMEuncentred} \end{cases}$}
		\end{tikzcd}
	}
	\caption{While conditioning of the probability distributions in the original spaces $\cX,\cY$ is a possibly complicated, non-linear problem, the corresponding formula for their kernel mean embeddings reduces to elementary linear algebra --- a common guiding theme when working with reproducing kernel Hilbert spaces.}
	\label{fig:CommutativeDiagramCME}
\end{figure}
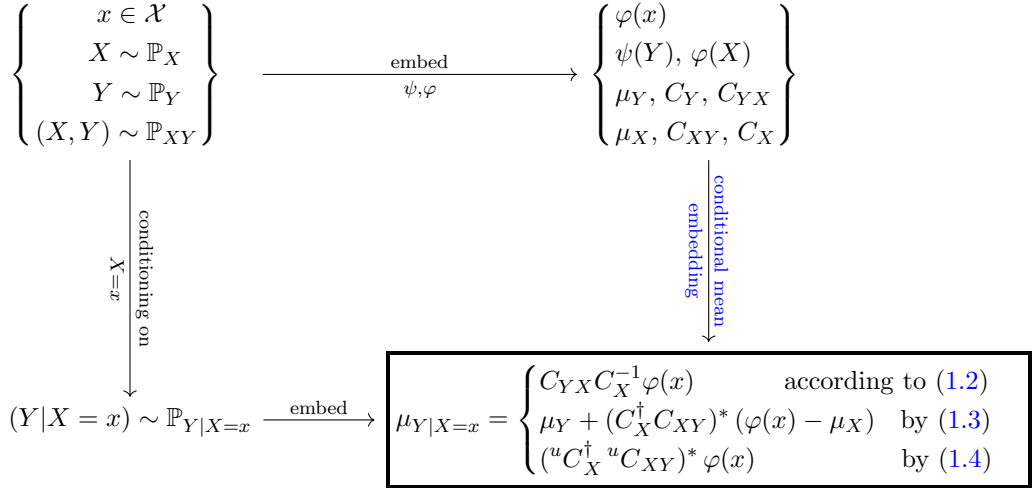

To be somewhat more precise --- while deferring technical points such as topological considerations, existence and uniqueness of conditional distributions etc.\ to \Cref{section:CMESetup} --- let us fix two RKHSs $\cH$ and $\cG$ over $\cX$ and $\cY$ respectively, with reproducing kernels $k$ and $\ell$ and canonical feature maps $\varphi(x) \defeq k(x,\quark)$ and $\psi(y) \defeq \ell(y,\quark)$.
Let $X$ and $Y$ be random variables taking values in $\cX$ and $\cY$ respectively with joint distribution $\bP_{XY}$ on $\cX \times \cY$.
Let $\mu_{X}$, $\mu_{Y}$, and $\mu_{Y|X = x}$ denote the kernel mean embeddings (KMEs) of the marginal distributions $\bP_{X}$ of $X$, $\bP_Y$ of $Y$, and the conditional distribution $\bP_{Y|X=x}$ of $Y$ given $X = x$ given by
\begin{equation}
	\label{equ:KMEandCME}
	\mu_{X}
	\defeq
	\bE[\varphi(X)] \in \cH,
	\qquad
	\mu_{Y}
	\defeq
	\bE[\psi(Y)] \in \cG,
	\qquad
	\mu_{Y|X=x}
	\defeq
	\bE[\psi(Y)|X=x] \in \cG.
\end{equation}
The \emph{conditional mean embedding} (CME) offers a way to perform conditioning of probability distributions on $\cX$ and $\cY$ by means of linear algebra in the corresponding feature spaces $\cH$ and $\cG$ (\Cref{fig:CommutativeDiagramCME}).
In terms of the kernel covariance operator $C_{X}$ and cross-covariance operator $C_{YX}$ defined later in \eqref{equ:MeanAndCovarianceBlockStructure}, if $C_{X}$ is invertible and $\bE[g(Y)|X=\quark]$ is an element of $\cH$ whenever $g \in \cG$, then the well-known formula for the CME \citep[Theorem 4]{song2009hilbert} is
\begin{align}
	\label{equ:CMEmu}
	\mu_{Y|X = x}
	&=
	C_{YX} C_{X}^{-1} \varphi(x),
	\qquad
	x\in\cX.
\end{align}
(We emphasise here that the CME $\mu_{Y|X = x}$ is \emph{defined} in \eqref{equ:KMEandCME} as the KME of $\bP_{Y|X=x}$;
the claim implicit in \eqref{equ:CMEmu} is that $\mu_{Y|X = x}$ can be \emph{realised} through simple linear algebra involving cross-covariance operators;
cf.\ the discussion of \citet{park2020measure}.)
Note that there are in fact two theories of CMEs, one working with \emph{centred} covariance operators \citep{fukumizu2004dimensionality,song2009hilbert} and the other with \emph{uncentred} ones \citep{fukumizu2013kernel}.
We will discuss both theories in detail, but let us focus for a moment on the centred case for which the above formula was originally derived.

In the trivial case where $X$ and $Y$ are independent, the CME should yield $\mu_{Y|X = x} = \mu_{Y}$.
However, independence implies that $C_{YX} = 0$, and so \eqref{equ:CMEmu} yields $\mu_{Y|X = x} = 0$, regardless of $x$.
In order to understand what has gone wrong it is helpful to consider in turn the two cases in which the constant function $\mathds{1}_\cX\colon x\mapsto 1$ is, or is not, an element of $\cH$.
\begin{itemize}
	\item If $\mathds{1}_{\cX} \in \cH$, then $C_{X}$ cannot be injective, since $C_{X}\mathds{1}_{\cX} = 0$, and \eqref{equ:CMEmu} is not applicable.	
	\item If $\mathds{1}_{\cX} \notin \cH$ and $X$ and $Y$ are independent, then the assumption $\bE[g(Y)|X=\quark]\in\cH$ for $g\in\cG$ cannot be fulfilled (except for those special elements $g\in\cG$ for which $\bE[g(Y)]=0$ or if $\bE[\ell(y,Y)] = 0$ for all $y\in\cY$, respectively), and \eqref{equ:CMEmu} is again not applicable.
\end{itemize}
In summary, \eqref{equ:CMEmu} is never applicable for independent random variables except in certain degenerate cases.
Note that this problem does not occur in the case of uncentred operators, where $\uu{C}_{X}$ (defined in \eqref{equ:UncentredCovarianceBlockStructure}) is typically injective.

Therefore, this paper aims to provide a rigorous theory of CMEs that addresses not only the above-mentioned pathology but also substantially generalises the assumptions under which CME can be performed.
We will treat both centred and uncentred \mbox{(cross-)}\-covariance operators, with particular emphasis on the centred case, and will also exhibit a connection to Gaussian conditioning in general Hilbert spaces.

\begin{enumerate}[label=(\arabic*)]
	\item The standard assumption $\bE[g(Y)|X=\quark]\in\cH$ for CME is rather restrictive.\footnote{\citet{fukumizu2013kernel} themselves write ``Note, however, that the assumptions [\dots] may not hold in general;
	we can easily give counterexamples for the latter in the case of Gaussian kernels.''.
	More precisely, for a Gaussian kernel $k$ on, say, $[0,1]$ and independent random variables $X$ and $Y$, $\bE[g(Y)|X=\quark]$ is a constant function for each $g\in\cG$, which does not lie in the RKHS corresponding to $k$ (unless it happens to be the zero function) by \citet[Corollary 5]{steinwart2006} or \citet[Corollary 4.44]{steinwart2008support}.}
	We show in \Cref{section:TheoryCentred} that this assumption can be significantly weakened in the case of centred kernel \mbox{(cross-)}\-covariance operators as defined in \eqref{equ:MeanAndCovarianceBlockStructure}:
	only $\bE[g(Y)|X=\quark]$ shifted by some constant function needs to lie in $\cH$ (\Cref{assump:weakerCME}).
	In this setting, the correct expression of the CME formula is
	\begin{equation}
		\label{equ:myCME}
		\mu_{Y|X = x}
		=
		\mu_{Y} + (C_{X}^{\dagger} C_{XY})^{\ast} \, (\varphi(x) - \mu_{X})
		\qquad
		\text{for $\bP_{X}$-a.e.\ $x\in\cX$,}
	\end{equation}
	where $A^{\ast}$ denotes the adjoint and $A^{\dagger}$ the Moore--Penrose pseudo-inverse of a linear operator $A$.
	As a first sanity check, note that this formula indeed yields $\mu_{Y|X = x} = \mu_{Y}$ when $X$ and $Y$ are independent.
	Similarly, as shown in \Cref{section:UncentredOperators}, for \emph{uncentred} kernel \mbox{(cross-)}\-covariance operators $\uu{C}_{X}$ and $\uu{C}_{XY}$ as defined later in \eqref{equ:UncentredCovarianceBlockStructure}, the correct formulation of the CME is
	\begin{equation}
		\label{equ:myCMEuncentred}
		\mu_{Y|X = x}
		=
		(\uu{C}_{X}^{\dagger} \uu{C}_{XY})^{\ast} \, \varphi(x)
		\qquad
		\text{for $\bP_{X}$-a.e.\ $x\in\cX$.}
	\end{equation}
	\item Furthermore, the assumption $\bE[g(Y)|X=\quark]\in\cH$, $g\in\cG$, is hard to check in most applications.
	To the best of our knowledge, the only verifiable condition that implies this assumption is given by \citet[Proposition 4]{fukumizu2004dimensionality_erratum}.
	However, this condition is itself difficult to check\footnote{The original condition of \citet[Proposition 4]{fukumizu2004dimensionality} was verifiable in certain situations, but the proposition itself turned out to be incorrect.
	The corrected condition in the erratum \citep{fukumizu2004dimensionality_erratum} seems to be much harder to check --- at least, no explicit case is given in which it is easier to verify than $\bE[g(Y)|X=\quark]$ being in $\cH$ for each $g\in\cG$. \label{footnote:FukumizuCondition}}.	
	We will present weaker assumptions (\Cref{assump:weakerCMElimit}) for the applicability of CMEs that hold whenever the kernel $k$ is characteristic.\footnote{A kernel $k$ is called \emph{characteristic} \citep{fukumizu2008nips} if the kernel mean embedding is injective as a function from $\{ \bQ \mid \bQ \text{ is a prob.\ meas.\ on } \cX \text{ with } \int_{\cX} \norm{ \varphi(x) }_{\cH} \, \rd \bQ (x) < \infty \}$ into $\cH$;
	naturally, the KME cannot be injective as a function from the space of random variables on $\cX$ to $\cH$, since random variables with the same law embed to the same point of $\cH$. \label{footnote:DefinitionCharacteristicKernel}}
	Characteristic kernels are well studied (see e.g.\ \citet{sriperumbudur2010hilbert}) and therefore provide a verifiable condition as desired.
	\item
	The applicability of \eqref{equ:CMEmu} requires the additional assumptions that $C_{X}$ is injective and that $\varphi(x)$ lies in the range of $C_{X}$,
	which is also hard to verify in practice.\footnote{Note that, typically, $\dim \cH = \infty$, in which case the compact operator $C_{X}$ cannot possibly be surjective.
	To verify that $\varphi(x)\in\ran C_{X}$, one would need to compute a singular value decomposition $C_{X} = \sum_{n \in \bN} \sigma_{n} h_{n} \otimes h_{n}$ of $C_{X}$ and check the Picard condition $\sum_{n \in \bN} \sigma_{n}^{-2} \innerprod{ \varphi(x) }{ h_{n} }_{\cH}^{2} < \infty$. \label{footnote:Picard}}
	We show that both assumptions can be avoided completely by replacing $C_{YX}C_{X}^{-1}$ in \eqref{equ:CMEmu} by $(C_{X}^{\dagger} C_{XY})^{\ast}$ in \eqref{equ:myCME} and $(\uu{C}_{X}^{\dagger} \uu{C}_{XY})^{\ast}$ in \eqref{equ:myCMEuncentred}, which turn out to be globally-defined and bounded operators under rather weak assumptions (Assumptions \ref{assump:weakCME} and \ref{assump:weakCMEuncentred}).
	\item
	The experienced reader will also observe that, modulo the replacement of $C_{YX}C_{X}^{-1}$ by $(C_{X}^{\dagger} C_{XY})^{\ast}$, \eqref{equ:myCME} is identical to the familiar Sherman--Morrison--Woodbury / Schur complement formula for conditional Gaussian distributions, a connection on which we will elaborate in detail in \Cref{section:ConnectionGaussianCME}.
	We call particular attention to the fact that the random variable $(\psi(Y),\varphi(X))$, which has no reason to be normally distributed, behaves very much like a Gaussian random variable in terms of its conditional mean.
\end{enumerate}

\begin{remark}
	\label{remark:AllResultsOnlyP_Xae}
	Note that we stated \eqref{equ:myCME} and \eqref{equ:myCMEuncentred} only for $\bP_X$-a.e.\ $x\in\cX$.
	This is the best that one can generally hope for, since the regular conditional probability $\bP_{Y|X=x}$ is uniquely determined only for $\bP_X$-a.e.\ $x\in\cX$ \citep[Theorem~5.3]{kallenberg2006foundations}.
	The work on CMEs so far completely ignores the fact that conditioning (especially on events of the form $X=x$) is not trivial, requires certain assumptions and, in general, yields results only for $\bP_X$-a.e.\ $x\in\cX$.
	In particular, the condition on $\bE[g(Y)|X=\quark]$ to lie in $\cH$ is ill posed, since these functions are uniquely defined only $\bP_{X}$-a.e., which in certain situations may be practically nowhere,
	and the same reasoning applies to the above-mentioned condition given by \citet[Proposition 4]{fukumizu2004dimensionality_erratum}.
	The existence and almost sure uniqueness of the regular conditional probability distribution $\bP_{Y|X=x}$ will be addressed in a precise manner in \Cref{section:CMESetup}.
\end{remark}

\begin{remark}
	\label{remark:EmpiricalIsHard}
	The focus of this paper is the validity of the non-regularised \emph{population} formulation of the CME in terms of the covariance structure of the KME of the data-generating distribution $\bP_{XY}$.
	The construction of valid CME formulae based on \emph{empirical sample data} (i.e.\ finitely many draws from $\bP_{XY}$) is vital in practice but is also much harder to analyse.
	We give some remarks on this setting in \Cref{section:empirical}.
\end{remark}

The rest of the paper is structured as follows.
\Cref{section:CMESetup} establishes the notation and problem setting, and motivates some of the assumptions that are made.
\Cref{section:AssumptionsForCMEs} discusses several critical assumptions for the applicability of the theory of CMEs and the relations among them.
\Cref{section:TheoryCentred} proceeds to build a rigorous theory of CMEs using centred covariance operators, with the main results being \Cref{thm:CMEproperly,thm:CMEproperlyLimit}, whereas \Cref{section:UncentredOperators} does the same for uncentred covariance operators, with the main results being \Cref{thm:CMEproperlyUncentered,thm:CMEproperlyUncenteredLimit}.
\Cref{section:GaussianConditioning} reviews the established theory for the conditioning of Gaussian measures on Hilbert spaces, and this is then used in \Cref{section:ConnectionGaussianCME} to rigorously connect the theory of CMEs to the conditioning of Gaussian measures, with the main result being \Cref{thm:ConnectionToGaussianRV}.
We give some closing remarks in \Cref{section:Closing}.
\Cref{section:TechnicalResults} contains various auxiliary technical results and \Cref{section:empirical} discusses the possible extension of our results to empirical estimation of CMEs.

%% file: section-02-setup.tex

\section{Setup and Notation}
\label{section:CMESetup}

Throughout this paper, when considering Hilbert-space valued random variables $U\in \cL^2(\Omega, \Sigma, \bP;\cG)$ and $V\in \cL^2(\Omega, \Sigma, \bP;\cH)$ defined over a probability space $(\Omega, \Sigma, \bP)$, the expected value $\bE[U] \defeq \int_{\Omega} U(\omega) \, \rd \bP (\omega)$ is meant in the sense of a Bochner integral \citep[Section II.2]{diestel1977}, as are the uncentred and centred cross-covariance operators
\begin{align*}
	\uu{\Cov}[U, V] \defeq \bE[U \otimes V]
	\qquad
	\text{and}
	\qquad
	\Cov[U, V] \defeq \bE[(U - \bE[U]) \otimes (V - \bE[V])]
\end{align*}
from $\cH$ into $\cG$, where, for $h \in \cH$ and $g \in \cG$, the outer product $g \otimes h \colon \cH \to \cG$ is the rank-one linear operator $(g \otimes h) (h') \defeq \innerprod{ h }{ h' }_{\cG}\, g$.
Naturally, we write $\uu{\Cov}[U]$ and $\Cov[U]$ for the covariance operators $\uu{\Cov}[U, U]$ and $\Cov[U, U]$ respectively, and all of the above reduces to the usual definitions in the scalar-valued case.
Both the centred and uncentred covariance operators of a square-integrable random variable are self-adjoint and non-negative, and --- in the separable Hilbert case that is our exclusive focus --- also trace-class (see \citet{baker1973joint, sazonov1958} for the centred case; the uncentred case follows from \citet[Corollary~2.1]{gohberg1969introduction}).

Our treatment of CMEs will operate under the following assumptions and notation:

\begin{assumption}
	\label{assumption:CME}
	\begin{enumerate}[label=(\alph*)]
		\item \label{notation:CMEsets} $(\Omega, \Sigma, \bP)$ is a probability space, $\cX$ is a measurable space, and $\cY$ is a Borel space.\footnote{A space $\cY$ is called a \emph{Borel space} if it is Borel isomorphic to a Borel subset of $[0,1]$.
		In particular, $\cY$ is a Borel space if it is Polish, i.e.\ if it is separable and completely metrisable;
		see \citet[Chapter~1]{kallenberg2006foundations}.}

		\item \label{notation:CMEkernels} $k \colon \cX \times \cX \to \bR$ and $\ell \colon \cY \times \cY \to \bR$ are symmetric and positive definite kernels, such that $k(x,\quark)$ and $\ell(y,\quark)$ are Borel-measurable functions for each $x\in\cX$ and $y\in\cY$.

		\item \label{notation:CMERKHS} $(\cH, \innerprod{ \quark }{ \quark }_{\cH})$ and $(\cG, \innerprod{ \quark }{ \quark }_{\cG})$ are the corresponding RKHSs, which we assume to be separable.
		Indeed, according to \citet{owhadi2017separability}, if the base sets $\cX$ and $\cY$ are separable absolute Borel spaces or analytic subsets of Polish spaces, then separability of $\cH$ and $\cG$ follows from the measurability of their respective kernels and feature maps.

		\item \label{notation:CMEfeaturemaps} The corresponding canonical feature maps are $\varphi \colon \cX \to \cH$, $\varphi(x) \defeq k(x,\quark)$, and $\psi \colon \cY \to \cG$, $\psi(y) \defeq \ell(y,\quark)$ respectively.
		Note that they satisfy the ``reproducing properties'' $\innerprod{ h }{ \varphi(x) }_{\cH} = h(x),\ \innerprod{ g }{ \psi(y) }_{\cG} = g(y)$ for $x\in\cX,y\in\cY,h\in\cH,g\in\cG$ and that $\varphi$ and $\psi$ are Borel measurable in view of \citet[Lemma~4.25]{steinwart2008support}.

		\item \label{notation:CMErvs} $X \colon \Omega \to \cX$ and $Y \colon \Omega \to \cY$ are random variables with distributions $\bP_{X}$ and $\bP_Y$ and joint distribution $\bP_{XY}$.
		\Cref{assumption:CME}\ref{notation:CMEsets} and \citet[Theorem~5.3]{kallenberg2006foundations} ensure the existence of a $\bP_{X}$-a.e.-unique regular version of the conditional probability distribution $\bP_{Y|X=x}$;
		the choice of a representative of $\bP_{Y|X=x}$ has no impact on our results.
		We assume that
		\begin{equation}
			\label{equ:FiniteSecondMomentOfUandV}
			\bE\big[\norm{ \varphi(X) }_{\cH}^{2} + \norm{ \psi(Y) }_{\cG}^{2} \big] < \infty,
		\end{equation}
		which also implies that $\cX_{Y} \defeq \{ x\in \cX \mid \bE\big[\norm{ \psi(Y) }_{\cG}^{2} | X=x\big] < \infty \}$ has full $\bP_X$ measure.\footnote{Otherwise, $\bE [\norm{ \psi(Y) }_{\cG}^{2}] = \bE [ \bE[\norm{ \psi(Y) }_{\cG}^{2} | X ] ]$ could not be finite.}
		Hence, $\cH\subseteq \cL^2(\bP_{X})$, $\cG\subseteq \cL^2(\bP_Y)$ and $\cG\subseteq \cL^2(\bP_{Y|X=x})$ for $x\in\cX_{Y}$ since, by the reproducing property and the Cauchy--Schwarz inequality,
		\begin{align}
			\notag
			\norm{ h }_{\cL^2(\bP_{X})}^2
			& =
			\int_{\cX} \absval{ h(x) }^2\, \rd \bP_{X}(x)
			=
			\int_{\cX} \absval{ \innerprod{ h }{ \varphi(x) }_{\cH} }^2\, \rd \bP_{X}(x) \\
			\label{equ:HnormStrongerThanL2norm}
			& \leq
			\int_{\cX} \norm{ h }_{\cH}^2 \norm{ \varphi(x) }_{\cH}^2\, \rd \bP_{X}(x)
			=
			\bE[\norm{ \varphi(X) }_\cH^2]\, \norm{ h }_\cH^2
		\end{align}
		for all $h\in\cH$, and similarly for $g\in\cG$ and $\bP_Y$, $\bP_{Y|X=x}$, $x \in \cX_{Y}$.
		It follows from \eqref{equ:HnormStrongerThanL2norm} that the inclusions $\iota_{\varphi,\bP_{X}}\colon \cH \hookrightarrow \cL^2(\bP_{X})$, $\iota_{\psi,\bP_Y}\colon \cG \hookrightarrow \cL^2(\bP_Y)$ are bounded linear operators and so is $\iota_{\psi,\bP_{Y|X=x}}\colon \cG \hookrightarrow \cL^2(\bP_{Y|X=x})$ for $x\in\cX_{Y}$.
		\item \label{notation:NoNontrivialZerosInH}
		We further assume that, for all $h\in\cH$, $h = 0$ $\bP_X$-a.e.\ in $\cX$ if and only if $h = 0$, i.e.\ almost everywhere equality separates points in $\cH$.
		This assumption clearly holds if $k$ is continuous and the topological support	of $\bP_{X}$ is all of $\cX$.\footnote{If $k$ is continuous, then so is every $h\in\cH$ \citep[Theorem 2.3]{saitoh2016theory}.
			So, if $h \in \cH$ and $\absval{ h(x) } = \varepsilon >0$ for some $x\in\cX$, then $\absval{ h } > \varepsilon/2$ on some open neighbourhood of $x$.
			Thus, if $\supp(\bP_X) = \cX$, then $h=0$ $\bP_X$-a.e.\ cannot hold.
		}
		It ensures that we can view $\cH$ as a subspace of $L^{2}(\bP_{X})$ and write $f\in\cH$ for functions $f\in L^2(\bP_{X})$ whenever there exists $h\in\cH$ (which, by this assumption, is unique) such that $f=h$ $\bP_{X}$-a.e.
		\item
		Several derivations will rely on the Bochner space $L^{2}(\bP_{X};\cF)$, which is isometrically isomorphic to the Hilbert tensor product space $L^{2}(\bP_{X})\otimes \cF$. Here, $\cF$ denotes another Hilbert space, which in our case will be equal to either $\bR$ or $\cG$.		
		Motivated by the discussion in \Cref{section:Intro} and the fact that $\Cov[f(X),f(X)] = \bV[f(X)]=0$ if and only if $f$ is $\bP_{X}$-a.e.\ constant, we consider the quotient space $L_{\cC}^{2}(\bP_{X}; \cF) \defeq L^{2}(\bP_{X}; \cF) / \cC$,\footnote{By the variational characterisation of the expected value $\bE[Z]$ of a random variable $Z\in L^{2}(\bP;\cF)$, $\bE[Z] = \argmin_{m\in\cF} \bE[\norm{Z-\bE[Z]}_{\cF}^{2}]$, the norm
		$\norm{ \quark }_{L_\cC^2(\bP_{X}; \cF)}$ coincides with the norm
		$\norm{ [f] } = \inf_{m\in\cC} \norm{ f-m }_{L^2(\bP_{X}; \cF)}$ induced on $L_{\cC}^2(\bP_{X}; \cF)$ by the norm $\norm{ \quark }_{L^2(\bP_{X}; \cF)}$.}
		\begin{align*}
			\cC
			\defeq
			\{ f\in L^{2}(\bP_{X}; \cF) \mid \exists c\in \cF\colon f(x) = c \text{ for $\bP_{X}$-a.e. } x\in\cX \},
			\\[0.6ex]
			\innerprod{ [f_1] }{ [f_2] }_{L_{\cC}^{2}(\bP_{X}; \cF)}
			\defeq \innerprod{ f_1 - \bE[f_1(X)] }{ f_2 - \bE[f_2(X)] }_{L^{2}(\bP_{X}; \cF)}.
		\end{align*}
		Note that, in the case $\cF = \bR$, we obtain $\innerprod{ [f_1] }{ [f_2]}_{L_{\cC}^{2}(\bP_{X}; \bR)} = \Cov[f_1(X),f_2(X)]$, in which case we will abbreviate the space $L^{2}(\bP_{X}; \bR)$ by $L^2(\bP_{X})$ or simply $L^2$ and the space $L_{\cC}^{2}(\bP_{X}; \bR)$ by $L_\cC^2(\bP_{X})$ or simply $L_\cC^2$.
		For any closed subspace $U\subseteq L^{2}(\bP_{X})$ we can view $U\otimes \cF$ as a subspace of $L^{2}(\bP_{X}; \cF)$ by the above isometry and identify $(U\otimes \cF)_{\cC} \defeq (U\otimes \cF) / ((U\otimes \cF) \cap \cC)$ with a subspace of $L_{\cC}^{2}(\bP_{X}; \cF)$.
		Note that, in the particular case $U\subseteq \cH$ (with $U$ closed in $L^{2}(\bP_{X})$), the construction of $U\otimes \cF$ and $(U\otimes \cF)_{\cC}$ treats $U$ as a subspace of $L^{2}(\bP_{X})$ and ignores the existence of the RKHS norm $\norm{\quark}_{\cH}$.
		
		\item We use overlines and superscripts to denote topological closures, so that, for example, $\overline{\cH_{\cC}}^{L_{\cC}^{2}}$ is the closure of $\cH_{\cC}$ with respect to the norm $\norm{ \quark }_{L_{\cC}^{2}}$ and $\overline{\cH}^{L^2}$ is the closure of $\cH$ with respect to the norm $\norm{ \quark }_{L^{2}}$.
		
		\item \label{notation:CMEmuC}
		Since $\varphi$ and $\psi$ are Borel measurable, $Z \defeq (\psi(Y),\varphi(X))$ is a well-defined $\cG\oplus\cH$-valued random variable;
		\eqref{equ:FiniteSecondMomentOfUandV} ensures that $Z$ has finite second moment, hence its mean $\bE[Z]$ and covariance operator $\Cov[Z]$ are well defined, Sazonov's theorem implies that $\Cov[Z]$ has finite trace, and we obtain the following block structures:
		\begin{equation}
			\label{equ:MeanAndCovarianceBlockStructure}
			\mu
			\defeq
			\bE\left[\begin{pmatrix}
			\psi(Y)\\ \varphi(X)
			\end{pmatrix}
			\right]
			=
			\begin{pmatrix}
			\mu_{Y}\\ \mu_{X}
			\end{pmatrix}
			,
			\qquad
			C
			\defeq
			\Cov\left[\begin{pmatrix}
			\psi(Y)\\ \varphi(X)
			\end{pmatrix}
			\right]
			=
			\begin{pmatrix}
			C_{Y} & C_{YX}
			\\
			C_{XY} & C_{X}
			\end{pmatrix} ,
		\end{equation}
		where the components
		\begin{align*}
			\mu_{Y}
			& \defeq
			\bE[\psi(Y)],&
			C_{Y}
			& \defeq
			\Cov[\psi(Y)],&
			C_{YX}
			& \defeq
			\Cov[\psi(Y),\varphi(X)],
			\\
			\mu_{X}
			& \defeq
			\bE[\varphi(X)],&
			C_{XY}
			& \defeq
			\Cov[\varphi(X),\psi(Y)],&
			C_{X}
			& \defeq
			\Cov[\varphi(X)]
		\end{align*}
		are called the \emph{kernel mean embeddings} (KME) and \emph{kernel \mbox{(cross-)}\-covariance operators}, respectively.
		Note that $C_{XY}^{\ast} = C_{YX}$ and that the reproducing properties translate to the KMEs and covariance operators as follows:
		for arbitrary $h,h'\in\cH$ and $g\in\cG$,
		\begin{align*}
			\innerprod{ h }{ \mu_{X} }_{\cH}
			& =
			\bE[h(X)] , \\
			\innerprod{ h }{ C_{X} h' }_{\cH}
			& =
			\Cov[h(X),h'(X)] , \\
			\innerprod{ h }{ C_{XY} g }_{\cH}
			& =
			\Cov[h(X),g(Y)] ,
		\end{align*}
		and so on.
		We are further interested in the \emph{conditional kernel mean embedding} and the \emph{conditional kernel covariance operator} given by
		\begin{equation}
			\label{equ:MeanAndCovarianceConditional}
			\mu_{Y|X=x}
			=
			\bE[\psi(Y)|X=x],
			\qquad
			C_{Y|X=x}
			=
			\Cov[\psi(Y)|X=x],
			\qquad
			x\in \cX_{Y}.
		\end{equation}
		We set $\mu_{Y|X=x} \defeq 0$ on the $\bP_{X}$-null set $\cX\setminus \cX_{Y}$.
		Similarly, $Z = (\psi(Y),\varphi(X))$ has the uncentred kernel covariance structure
		\begin{equation}
			\label{equ:UncentredCovarianceBlockStructure}
			\uu{C}
			\defeq
			\uu{\Cov}\left[\begin{pmatrix}
			\psi(Y)\\ \varphi(X)
			\end{pmatrix}
			\right]
			=
			\begin{pmatrix}
			\uu{C}_{Y} & \uu{C}_{YX}
			\\
			\uu{C}_{XY} & \uu{C}_{X}
			\end{pmatrix} ,
		\end{equation}
		where $\uu{C}_{Y} \defeq \uu{\Cov}[\psi(Y)]$ etc.
		Note that, for $f_1,f_2\in L^2(\bP_{X})$, $\uu\Cov(f_1(X),f_2(X)) = \innerprod{ f_{1} }{ f_{2} }_{L^2(\bP_{X})}$, and similarly for functions of $Y$.
		
		\item For $g\in\cG$ we let $f_g(x) \defeq \bE[g(Y)| X=x]$.
		More precisely,
		\[
			f_g(x) \defeq
			\begin{cases}
			\bE[g(Y)| X=x], & \text{for } x\in\cX_{Y},\\
			0, & \text{otherwise.}
			\end{cases}
		\]
		These functions $f_{g}$ will be of particular importance since, for $g=\psi(y)$, $y\in\cY$ and $x\in\cX$, we obtain $f_{\psi(y)}(x) = \mu_{Y|X = x}(y)$, our main object of interest (note that $\mu_{Y|X = x}\in\cG$ for each $x\in\cX$, and so its pointwise evaluation at $y\in\cY$ is meaningful).
		By \eqref{equ:FiniteSecondMomentOfUandV}, \eqref{equ:HnormStrongerThanL2norm}, and the law of total expectation, $f_g\in L^2(\bP_{X})$ for every $g\in\cG$, since
		\begin{align*}
			\norm{ f_g }_{L^2(\bP_{X})}
			& =
			\bE[f_g(X)^2]
			=
			\bE\left[\bE[g(Y)|X]^2\right]
			\leq
			\bE\left[\bE[g(Y)^2|X]\right] \\
			& =
			\bE[g(Y)^2]
			=
			\norm{ g }_{\cL^2(\bP_Y)}
			<
			\infty.
		\end{align*}
		Further, another application of the law of total expectation yields
		\begin{equation}
		\label{equ:f_gExpectation}
		\bE[f_g(X)] = \bE[g(Y)],
		\qquad
		\bE[f_{\psi(y)}(X)] = \mu_{Y}(y).
		\end{equation}
		
%

		\item For a linear operator $A$ between Hilbert spaces, $A^{\dagger}$ denotes its Moore--Penrose pseudo-inverse, i.e.\ the unique extension of $A |_{(\ker A)^{\perp}}^{-1} \colon \ran A \to (\ker A)^{\perp}$ to a linear operator $A^{\dagger}$ defined on $\dom A^{\dagger} \defeq (\ran A) \oplus (\ran A)^{\perp}$ subject to the criterion that $\ker A^{\dagger} = (\ran A)^{\perp}$.
		In general, $\dom A^{\dagger}$ is a dense but proper subpace and $A^{\dagger}$ is an unbounded operator;
		global definition and boundedness occur precisely when $\ran A$ is closed;
		see e.g.\ \citet[Section~2.1]{engl1996regularization}.
	\end{enumerate}
\end{assumption}

\begin{remark}
	Measurability of $k(x,\quark)$ and $\ell(y,\quark)$ together with the separability of $\cH$ and $\cG$ guarantee the measurability of $\varphi$ and $\psi$ \citep[Lemma~4.25]{steinwart2008support}.
	Separability of $\cH$ and $\cG$ is also needed for Gaussian conditioning (see \citet{owhadi2015conditioning} and \Cref{section:GaussianConditioning}), for the existence of a countable orthonormal basis of $\cH$, and to ensure that weak (Pettis) and strong (Bochner) measurability of Hilbert-valued random variables coincide.
\end{remark}

%% file: section-03-assumptions.tex

\section{The Crucial Assumptions for CMEs}
\label{section:AssumptionsForCMEs}

This section discusses various versions of the assumption $f_g\in\cH$ under which we are going to prove various versions of the CME formula (note that, by \Cref{assumption:CME}\ref{notation:NoNontrivialZerosInH}, their formulations are unambiguous).

\begin{assumpA}
	\label{assump:strongCME}
	For all $g\in\cG$, $f_g\in\cH$.
\end{assumpA}

\begin{assumpB}
	\label{assump:weakerCME}
	For all $g\in\cG$ there exists a function $h_g\in\cH$ and a constant $c_g\in\bR$ such that $h_g = f_g - c_g$ $\bP_{X}$-a.e.\ in $\cX$.
\end{assumpB}

\begin{assumpC}
	\label{assump:weakCME}
	For all $g\in\cG$ there exists a function $h_g\in\cH$ such that
	\[
	\Cov[h_g(X)-f_g(X),h(X)] = 0\qquad\text{for all }h\in\cH.
	\]
	In this case we denote $c_g \defeq \bE[f_g(X)-h_g(X)]$ (in conformity with \Cref{assump:weakerCME}).
\end{assumpC}

\begin{assumpCdoubleprime}
	\label{assump:weakCMEuncentred}
	For all $g\in\cG$ there exists a function $h_g\in\cH$ such that
	\[
		\uu\Cov[h_g(X)-f_g(X),h(X)] = \innerprod{ h_g-f_g }{ h }_{ L^2(\bP_{X})} = 0\qquad\text{for all }h\in\cH.
	\]
\end{assumpCdoubleprime}

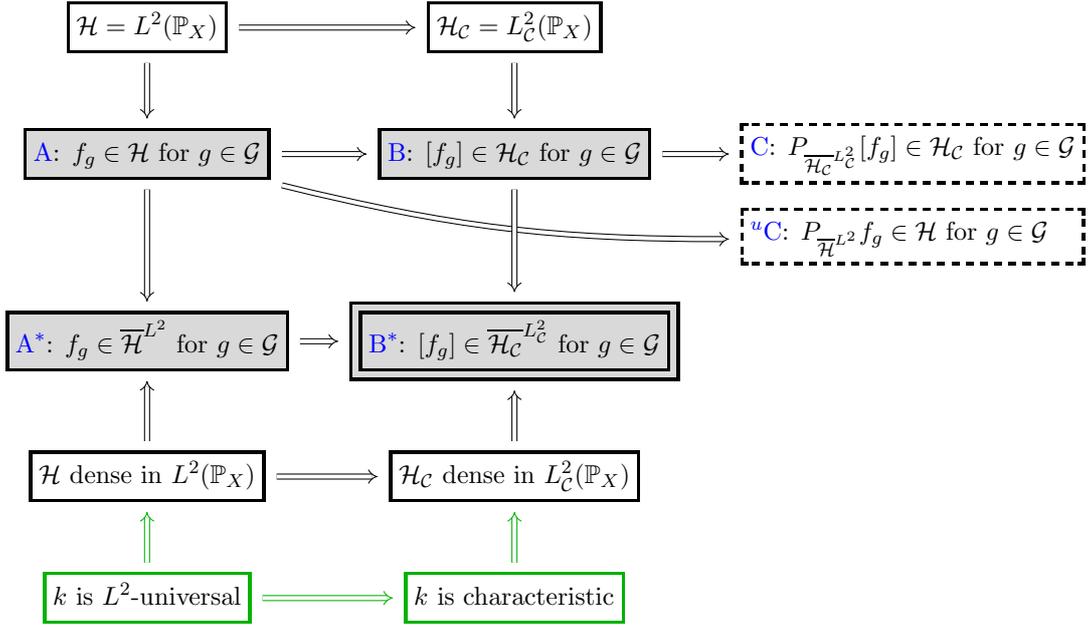
\begin{figure}[t]
	\centering
	\adjustbox{scale=0.89}{
		\begin{tikzcd}[column sep=1.5em,row sep=2em]	
			\mlnode{$\cH = L^2(\bP_{X})$}
			\arrow[r,Rightarrow]
			\arrow[d,Rightarrow]
			&
			\mlnode{$\cH_{\cC} = L_{\cC}^{2}(\bP_{X})$}
			\arrow[d,Rightarrow]
			&
			\\
			\filledmlnode{\ref{assump:strongCME}: $f_g\in\cH$ for $g\in\cG$}
			\arrow[r,Rightarrow]
			\arrow[dd,Rightarrow]
			\arrow[bend right=7,drr,Rightarrow]
			&
			\filledmlnode{\ref{assump:weakerCME}: $[f_g]\in\cH_{\cC}$ for $g\in\cG$}
			\arrow[r,Rightarrow]
			\arrow[dd,Rightarrow]
			&
			\dashedmlnode{\ref{assump:weakCME}: $P_{\overline{\cH_{\cC}}^{ L_{\cC}^{2}}}[f_g] \in \cH_{\cC}$ for $g\in\cG$}
			\\[-0.7cm]
			&
			&
			\dashedmlnode{\ref{assump:weakCMEuncentred}: $P_{\overline{\cH}^{ L^2}}f_g \in \cH$ for $g\in\cG$\hspace{0.4cm}\ }
			\\[-0.5cm]
			\filledmlnode{\ref{assump:strongCMElimit}: $f_{g}\in\overline{\cH}^{ L^2}$ for $g\in\cG$}
			\arrow[r,Rightarrow]
			&
			\filledmlnode{\filledmlnode{\ref{assump:weakerCMElimit}: $[f_{g}]\in
				\overline{\cH_{\cC}}^{ L_{\cC}^2}$ for $g\in\cG$			
			}}
			\arrow[d,Leftarrow]
			&
			\\
			\mlnode{$\cH \text{ dense in } L^2(\bP_{X})$}
			\arrow[r,Rightarrow]
			\arrow[u,Rightarrow]
			&
			\mlnode{$\cH_{\cC} \text{ dense in } L_{\cC}^{2}(\bP_{X})$}
			\arrow[d,Leftarrow,boxgreen]
			&
			\\
			\greenmlnode{$k$ is $ L^2$-universal}
			\arrow[u,Rightarrow,boxgreen]
			\arrow[r,Rightarrow,boxgreen]			
			&
			\greenmlnode{$k$ is characteristic}
			&			
		\end{tikzcd}
	}
	\caption{A hierarchy of CME-related assumptions.		
	Sufficient conditions for validity of the CME formula are indicated by solid boxes while the insufficient Assumptions \ref{assump:weakCME} and \ref{assump:weakCMEuncentred}, indicated by dashed boxes, have several strong theoretical implications and \Cref{assump:weakCME} has a beautiful connection to Gaussian conditioning (\Cref{thm:EquivalenceAssumptionCAndCompatibility}).			
	\Cref{assump:weakerCMElimit} is the most favorable one, since it is verifiable in practice, and, by \Cref{lemma:EquivalenceCharacteristicDense}, in particular is fulfilled if the kernel is universal or even just characteristic (marked in green).
	The shaded boxes correspond to \Cref{thm:CMEproperly,thm:CMEproperlyLimit,thm:CMEproperlyUncentered,thm:CMEproperlyUncenteredLimit}.}	
	\label{fig:HierarchyOfAssumptions}
\end{figure}

\begin{remark}
	Note that \ref{assump:strongCME} $\implies$ \ref{assump:weakerCME} $\implies$ \ref{assump:weakCME}, that \ref{assump:strongCME} $\implies$ \ref{assump:weakCMEuncentred}, that \ref{assump:weakCME} $\implies$ \ref{assump:weakerCME} if $\cH_{\cC}\subseteq L_{\cC}^2(\bP_{X})$ is dense, and that \ref{assump:weakCME} $\implies$ \ref{assump:strongCME} and \ref{assump:weakCMEuncentred} $\implies$ \ref{assump:strongCME} if $\cH\subseteq L^2(\bP_{X})$ is dense.
\end{remark}

Unlike \Cref{assump:strongCME}, Assumptions \ref{assump:weakerCME} and \ref{assump:weakCME} do not require the unfavourable property $\mathds{1}_\cX\in\cH$ for independent random variables $X$ and $Y$.
Instead, this case reduces to the trivial condition $0\in\cH$.
At the same time, the proofs of the key properties of CMEs are not affected by replacing \Cref{assump:strongCME} with \Cref{assump:weakerCME} as long as we work with centred operators (see \Cref{thm:RelationC_XandC_XY,thm:CMEproperly} below).
Therefore, it is surprising that this modification has not been considered earlier, even though the issues with independent random variables have been observed before \citep{fukumizu2013kernel}.
One reason might be that, instead of centred operators, researchers started using uncentred ones, for which such a modification is not feasible.

\Cref{assump:weakCME}, on the other hand, is not strong enough for proving the main formula for CMEs (the last statement of \Cref{thm:CMEproperly}).
Clearly, this cannot be expected:
If $\cX$ and $\cG$ are reasonably large, but $\cH$ is not rich enough, e.g.\ $\cH = \{ 0 \}$ or $\cH=\spn \{ \mathds{1}_\cX \}$, then no map from $\cH$ to $\cG$ can cover sufficiently many kernel mean embeddings,
in particular the embeddings of the conditional probability $\bP_{Y|X=x}$ for various $x$ (while \Cref{assump:weakCME} is trivially fulfilled for these choices of $\cH$).
The weakness of \Cref{assump:weakCME} lies in the fact that it only requires the vanishing of the orthogonal projection of $[h_g]-[f_g]$ onto $\cH_{\cC}$.
Only if $\cH_{\cC}$ is rich enough (e.g.\ if it is dense in $ L_{\cC}^2$) can this condition have useful implications.
A similar reasoning applies to \Cref{assump:weakCMEuncentred}.

While it is nice to have a weaker form of \Cref{assump:strongCME}, the Assumptions \ref{assump:strongCME}, \ref{assump:weakerCME} and \ref{assump:weakCME} remain hard to check in practice.
Another condition, provided by \citet[Proposition 4]{fukumizu2004dimensionality_erratum}, is also hard to verify in most applications.\textsuperscript{\ref{footnote:FukumizuCondition}}
Since characteristic kernels are well studied in the literature, \Cref{lemma:EquivalenceCharacteristicDense} gives hope for a verifiable condition for the applicability of CMEs:
it states that $\cH_{\cC}$ is dense in $ L_{\cC}^{2}(\bP_{X})$ whenever the kernel $k$ is characteristic.
So, if the denseness of $\cH_{\cC}$ in $ L_{\cC}^{2}(\bP_{X})$ were sufficient for performing CMEs, then the condition that $k$ be characteristic would be sufficient as well, thus providing a favorable criterion for the applicability of formula \eqref{equ:myCME}.
A similar argumentation applies to the condition that $\cH$ is dense in $ L^{2}(\bP_{X})$ and the condition that $k$ is $ L^2$-universal.\footnote{A kernel $k$ on $\cX$ is called \emph{$ L^2$-universal} \citep{sriperumbudur2011universality} if it is Borel measurable and bounded and if $\cH$ is dense in $ L^{2}(\bQ)$ for any probability measure $\bQ$ on $\cX$.
Any $ L^2$-universal kernel is characteristic \citep{sriperumbudur2011universality}.}
Unfortunately, neither condition implies \Cref{assump:weakerCME}.
Therefore, we will consider the following slightly weaker versions of Assumptions~\ref{assump:strongCME} and \ref{assump:weakerCME}, under which CMEs can be performed if one allows for certain finite-rank approximations of the \mbox{(cross-)}\-covariance operators:

\begin{assumpAprime}
	\label{assump:strongCMElimit}
	For all $g\in\cG$, $f_g\in\overline{\cH}^{ L^2}$.
\end{assumpAprime}

\begin{assumpBprime}
	\label{assump:weakerCMElimit}
	For all $g\in\cG$ there exists a function $h_g\in\overline{\cH}^{ L^2}$ and a constant $c_g\in\bR$ such that $h_g = f_g - c_g$ $\bP_X$-a.e.\ in $\cX$.
\end{assumpBprime}


Note that \Cref{assump:weakCME} and \Cref{assump:weakCMEuncentred} have no weaker versions, since they would become trivial if $h_g\in\cH_{\cC}$ were replaced by $h_g\in\overline{\cH_{\cC}}^{ L_{\cC}^{2}}$ and $h_g\in\cH$ by $h_g\in\overline{\cH}^{ L^2}$ respectively.

\begin{remark}
	\label{remark:reformulationA1A2A3}
	In terms of the spaces $ L_{\cC}^2$ and $\cH_{\cC}$, Assumptions \ref{assump:strongCME}--\ref{assump:weakerCMElimit} can be reformulated as follows:
	For all $g \in \cG$,
	\begin{itemize}
		\item[(\ref{assump:strongCME})]
		$f_{g} \in\cH$;
		\item[(\ref{assump:weakerCME})]
		$[f_{g}]\in\cH_{\cC}$;
		\item[(\ref{assump:weakCME})]
		the orthogonal projection $P_{\overline{\cH_{\cC}}^{ L_{\cC}^{2}}} [f_{g}]$ of $[f_{g}]$ onto $\overline{\cH_{\cC}}^{ L_{\cC}^{2}}$ lies in $\cH_{\cC}$;
		\item[(\ref{assump:weakCMEuncentred})]
		the orthogonal projection $P_{\overline{\cH}^{ L^2}} f_{g}$ of $f_{g}$ onto $\overline{\cH}^{L^2}$ lies in $\cH$;
		\item[(\ref{assump:strongCMElimit})]
		$f_{g} \in \overline{\cH}^{L^2}$;
		\item[(\ref{assump:weakerCMElimit})]
		$[f_{g}] \in \overline{\cH_{\cC}}^{L_{\cC}^2}$.
	\end{itemize}
\end{remark}

In summary, we consider the hierarchy of assumptions illustrated in \Cref{fig:HierarchyOfAssumptions}.
The main contributions of this paper are rigorous proofs of three versions of the CME formula under various assumptions:
\begin{itemize}
	\item \Cref{thm:CMEproperly} uses \Cref{assump:weakerCME} and centred operators.
	\item \Cref{thm:CMEproperlyLimit} uses \Cref{assump:weakerCMElimit} and finite-rank approximations of centred operators.
	\item \Cref{thm:CMEproperlyUncentered} uses \Cref{assump:strongCME} and uncentred operators.
	\item \Cref{thm:CMEproperlyUncenteredLimit} uses \Cref{assump:strongCMElimit} and finite-rank approximations of uncentred operators.
\end{itemize}
Note that the theorems for uncentred covariance operators require stronger assumptions than their centred counterparts and that \Cref{thm:CMEproperlyUncenteredLimit} provides weaker statements than its centred analogue \Cref{thm:CMEproperlyLimit} (we show only the convergence in $L^2(\bP_{X};\cG)$, which does not guarantee convergence for $\bP_{X}$-a.e.\ $x\in\cX$).

%% file: section-04-centred.tex

\section{Theory for Centred Operators}
\label{section:TheoryCentred}

In this section we will formulate and prove two versions of the CME formula \eqref{equ:myCME} --- the original one under \Cref{assump:weakerCME} and a weaker version involving finite-rank approximations $C_{X}^{(n)},C_{XY}^{(n)}$ of the \mbox{(cross-)}\-covariance operators under \Cref{assump:weakerCMElimit}.
The following theorem demonstrates the importance of \Cref{assump:weakCME} (which follows from \Cref{assump:weakerCME}).
It implies that the range of $C_{XY}$ is contained in that of $C_{X}$, making the operator $C_{X}^\dagger C_{XY}$ well defined.
By \Cref{theorem:DouglasExistenceQ} it is even a bounded operator, which is a non-trivial result requiring the application of the closed graph theorem.\footnote{Furthermore, by \Cref{lemma:Carleman_trick}, this operator is actually Hilbert--Schmidt when thought of as an operator taking values in the appropriate $L^{2}$ space rather than in the RKHS. \label{footnote:Carleman_trick}}

Similar considerations cannot be performed, in general, under \Cref{assump:weakerCMElimit} alone: it can no longer be expected that $\ran C_{XY} \subseteq \ran C_{X}$, which is why we must introduce the above-mentioned finite-rank approximations in order to guarantee that $\ran C_{XY}^{(n)} \subseteq \ran C_{X}^{(n)}$.

In summary, \Cref{assump:weakerCME} allows for the simple CME formula \eqref{equ:myCME} by \Cref{thm:RelationC_XandC_XY}, while under \Cref{assump:weakerCMElimit} we must make a detour using certain approximations.
Note that this distinction is very similar to the theory of Gaussian conditioning in Hilbert spaces introduced by \citet{owhadi2015conditioning} and recapped in \Cref{section:GaussianConditioning} below, a connection that will be elaborated upon in detail in \Cref{section:ConnectionGaussianCME}.


\begin{theorem}
	\label{thm:RelationC_XandC_XY}%
	Under \Cref{assumption:CME}, the following statements are equivalent:
	\begin{enumerate}[label=(\roman*)]
		\item \label{statement1lemma:RelationC_XandC_XY} \Cref{assump:weakCME} holds.
		\item \label{statement2lemma:RelationC_XandC_XY} For each $g\in\cG$ there exists $h_g\in\cH$ such that
		$
		C_{X} h_g
		=
		C_{XY}g.
		$
		\item \label{statement3lemma:RelationC_XandC_XY} $\ran C_{XY}\subseteq \ran C_{X}$.
	\end{enumerate}
\end{theorem}


\begin{proof}
	Note that \ref{statement3lemma:RelationC_XandC_XY} is just a reformulation of \ref{statement2lemma:RelationC_XandC_XY}, so we only have to prove \ref{statement1lemma:RelationC_XandC_XY}$\iff$\ref{statement2lemma:RelationC_XandC_XY}.
	Let $g\in\cG$ and $h,h_g\in\cH$.
	By \Cref{lemma:RelationCovAndC_XY}, $\Cov [h(X), f_g(X)]= \innerprod{ h }{ C_{XY} g }_{\cH}$, and so
	\begin{align*}
	\Cov [h(X), h_g(X)] = \Cov [h(X), f_g(X)]\ \forall h\in\cH
	&\iff
	\innerprod{ h }{ C_{X} h_g }_{\cH} = \innerprod{ h }{ C_{XY} g }_{\cH} \ \forall h\in\cH \\
	& \iff
	C_{X} h_g = C_{XY} g .
	\end{align*}
\end{proof}


Note that \Cref{assump:weakCME} implies that $[h_g]\in\cH_\cC$ is the orthogonal projection of $[f_g]\in L_{\cC}^{2}$ onto $\cH_\cC$ with respect to $\innerprod{ \quark }{ \quark }_{ L_{\cC}^{2}}$ (see the reformulation of \Cref{assump:weakCME} in \Cref{remark:reformulationA1A2A3}).
Therefore, there might be some ambiguity in the choice of $h_g\in\cH$ if $\cH$ contains constant functions.
However, there is a particular choice of $h_g$ that always works:


\begin{proposition}
	\label{prop:SpecificChoiceh_g}
	Under \Cref{assumption:CME}, if \Cref{assump:weakerCME} or \Cref{assump:weakCME} holds, then $h_g$ may be chosen as
	\begin{equation}
		\label{equ:StandardChoiceh_g}
		h_g = C_{X}^\dagger C_{XY}g.
	\end{equation}
	More precisely, if \Cref{assump:weakCME} holds, then $\Cov[(C_{X}^\dagger C_{XY}g)(X)-f_g(X),h(X)] = 0$ for all $h\in\cH$ and $g\in\cG$;
	and if \Cref{assump:weakerCME} holds, or even just $f_{g}\in\cH_{\cC}$ for some $g\in\cG$, then there exists a constant $c_g\in\bR$ such that $\bP_X$-almost everywhere $f_g = c_g + C_{X}^\dagger C_{XY}g$.
\end{proposition}


\begin{proof}
	By \Cref{thm:RelationC_XandC_XY}, \eqref{equ:StandardChoiceh_g} is well defined.
	Under \Cref{assump:weakCME}, for all $g\in\cG$ and $h\in\cH$, and appealing to \Cref{thm:RelationC_XandC_XY} and \Cref{lemma:RelationCovAndC_XY},
	\begin{align*}
		\Cov [h(X), (C_{X}^\dagger C_{XY}g)(X)]
		=
		\innerprod{ h }{ C_{X} C_{X}^\dagger C_{XY} g }_{\cH}
		=
		\innerprod{ h }{ C_{XY} g }_{\cH}
		= 
		\Cov [h(X), f_g(X)] .
	\end{align*}
	If $f_{g}\in\cH_{\cC}$ for some $g\in\cG$, then there exists a function $h_g'\in\cH$ and a constant $c_g'\in\bR$ such that, $\bP_X$-a.e.\ in $\cX$, $h_g' = f_g - c_g'$.
	\Cref{thm:RelationC_XandC_XY} implies that $C_{X} h_g'=C_{XY}g$, and so \Cref{lemma:kernelC_Xconstants} implies that $h_g' - C_{X}^\dagger C_{XY} g$ is constant $\bP_X$-a.e.
	Hence, $f_g - C_{X}^\dagger C_{XY} g$ is constant $\bP_X$-a.e.
\end{proof}


We now give our first main result, the rigorous statement of the CME formula for centred \mbox{(cross-)}\-covariance operators.
In fact, we give two results:
a ``weak'' result \eqref{equ:CMEmuNEWProjection} under \Cref{assump:weakCME} in which the CME, as a function on $\cX$, holds only when tested against elements of $\cH$ in the $ L^{2}(\bP_X)$ inner product, and a ``strong'' almost-sure equality in $\cG$ \eqref{equ:CMEmuNEW} under \Cref{assump:weakerCME}.


\begin{theorem}[Centred CME]
	\label{thm:CMEproperly}
	Under Assumptions \ref{assumption:CME} and \ref{assump:weakCME}, $C_{X}^{\dagger} C_{XY}\colon \cG\to\cH$ is a bounded\textsuperscript{\ref{footnote:Carleman_trick}} operator and, for all $y\in\cY$ and $h\in\cH$,
	\begin{align}
		\label{equ:CMEmuNEWProjection}
		\innerprod{ h }{ \mu_{Y|X = \quark}(y) }_{ L^{2}(\bP_X)}
		&=
		\Innerprod{ h }{ \big(\mu_{Y} + (C_{X}^{\dagger} C_{XY})^{\ast} \, (\varphi(\quark) - \mu_{X})\big)(y) }_{ L^{2}(\bP_X)}.
	\end{align}
	Suppose in addition that any of the following four conditions holds:
	\begin{enumerate}[label=(\roman*)]
		\item \label{item:CMEproperly1} the kernel $k$ is characteristic;
		\item \label{item:CMEproperly2} $\cH_{\cC}$ is dense in $ L_{\cC}^{2}(\bP_X)$;
		\item \label{item:CMEproperly3} \Cref{assump:weakerCME} holds;
		\item \label{item:CMEproperly4} $f_{\psi(y)}\in \cH_{\cC}$ for each $y\in\cY$.
	\end{enumerate}
	Then, for $\bP_X$-a.e.\ $x\in\cX$,
	\begin{align}
		\label{equ:CMEmuNEW}
		\mu_{Y|X = x}
		&=
		\mu_{Y} + (C_{X}^{\dagger} C_{XY})^{\ast} \, (\varphi(x) - \mu_{X}).
	\end{align}
\end{theorem}


\begin{proof}
	\Cref{thm:RelationC_XandC_XY,theorem:DouglasExistenceQ} imply that $C_{X}^{\dagger} C_{XY}$ is well defined and bounded\textsuperscript{\ref{footnote:Carleman_trick}} and that, for each $g\in\cG$, we may choose the function $h_g\in\cH$ in Assumptions \ref{assump:weakerCME} and \ref{assump:weakCME} to be $h_g = C_{X}^{\dagger} C_{XY}g$ (by \Cref{prop:SpecificChoiceh_g}).
	Now \eqref{equ:f_gExpectation}, \Cref{lemma:TeproducingPropertyAdjoint}, and the definition of $c_g$ (see \Cref{assump:weakCME}) yield that, for $x\in\cX$ and $y\in\cY$,
	\begin{equation}
		\label{equ:RewritingH}
		h_{\psi(y)}(x) + c_{\psi(y)}
		=
		\big(\mu_{Y} + (C_{X}^{\dagger} C_{XY})^{\ast} (\varphi(x)-\mu_{X})\big)(y) .
	\end{equation}
	This yields \eqref{equ:CMEmuNEWProjection} for each $h\in\cH$ via
	\begin{align*}
		& \innerprod{ h }{ \big(\mu_{Y|X = \quark} - \mu_{Y} - (C_{X}^{\dagger} C_{XY})^{\ast} \, (\varphi(\quark) - \mu_{X})\big)(y) }_{L^{2}(\bP_X)} \\
		& \quad =
		\innerprod{ h }{ f_{\psi(y)} - h_{\psi(y)} - c_{\psi(y)} }_{ L^{2}(\bP_X)} \\
		& \quad =
		\underbrace{\Cov[h(X),(f_{\psi(y)} - h_{\psi(y)})(X)]}_{=0}
		+
		\bE[h(X)]\big(\underbrace{\bE[(f_{\psi(y)}- h_{\psi(y)})(X)] - c_\psi(y)}_{=0}\big)
		=0.
	\end{align*}
	If \ref{item:CMEproperly1} or \ref{item:CMEproperly2} holds (note that, by \Cref{lemma:EquivalenceCharacteristicDense}, \ref{item:CMEproperly1}$\implies$\ref{item:CMEproperly2}), then \eqref{equ:CMEmuNEW} follows directly.
	If \ref{item:CMEproperly3} or \ref{item:CMEproperly4} holds (with $f_{\psi(y)} = h_{\psi(y)} + c_{\psi(y)}$, $h_{\psi(y)}\in\cH$, $c_{\psi(y)}\in\bR$), then \eqref{equ:CMEmuNEW} can be obtained from
	\begin{align*}
		\mu_{Y|X = x}(y)
		& =
		\bE[\ell(y,Y)|X=x]
		=
		f_{\psi(y)}(x)
		\stackrel{(\ast)}{=}
		h_{\psi(y)}(x) + c_{\psi(y)} \\
		& =
		\big(\mu_{Y} + (C_{X}^{\dagger} C_{XY})^{\ast} \, (\varphi(x) - \mu_{X})\big)(y) ,
	\end{align*}
	where all equalities hold for $\bP_X$-a.e.\ $x\in\cX$ and the last equality follows from \eqref{equ:RewritingH} (note that we might be arguing with two different choices of $h_{\psi(y)}$, which me may assume to agree by \Cref{prop:SpecificChoiceh_g}).
\end{proof}


Note that step $(\ast)$ in the proof of \Cref{thm:CMEproperly} genuinely requires condition \ref{item:CMEproperly4} (which follows from \Cref{assump:weakerCME}), and \Cref{assump:weakCME} alone does not suffice.
Again we see that $\cH$ needs to be rich enough.
The reason that we get \eqref{equ:CMEmuNEWProjection} in terms of the inner product of $ L^{2}(\bP_X)$, and not its weaker version in $ L_{\cC}^{2}(\bP_X)$, is that we took care of the shifting constant $c_g \defeq \bE[f_g(X) - h_g(X)]$.

Motivated by the theory of Gaussian conditioning in Hilbert spaces \citep{owhadi2015conditioning} presented in \Cref{section:GaussianConditioning} and \Cref{thm:MainTheoremOwhadiScovel} in particular, we hope to generalise CMEs to the case where $\ran C_{XY}\subseteq \ran C_{X}$ (i.e., by \Cref{thm:RelationC_XandC_XY}, \Cref{assump:weakCME}) does not necessarily hold.
As mentioned above, this will require us to work with certain finite-rank approximations of the operators $C_{X}$ and $C_{XY}$.
We are still going to need some assumption that guarantees that $\cH$ is rich enough to be able to perform the conditioning process in the RKHSs.
For this purpose \Cref{assump:weakerCME} will be replaced by its weaker version \ref{assump:weakerCMElimit}.


\begin{theorem}[Centred CME under finite-rank approximation]
	\label{thm:CMEproperlyLimit}
	Let \Cref{assumption:CME} hold.
	Further, let $(h_n)_{n\in\bN}$ be a complete orthonormal system of $\cH$ that is an eigenbasis of $C_{X}$, let $\cH^{(n)} \defeq \spn \{ h_1,\dots,h_n \}$, let $\cF \defeq \cG\oplus\cH$, let $P^{(n)}\colon\cF\to\cF$ be the orthogonal projection onto $\cG\oplus \cH^{(n)}$, and let
	\[
		C
		\defeq
		\begin{pmatrix}
		C_{Y} & C_{YX}
		\\
		C_{XY} & C_{X}
		\end{pmatrix},
		\qquad
		C^{(n)}
		\defeq
		P^{(n)}CP^{(n)}
		=
		\begin{pmatrix}
		C_{Y} & C_{YX}^{(n)}
		\\
		C_{XY}^{(n)} & C_{X}^{(n)}
		\end{pmatrix}.
	\]
	Then $\ran C_{XY}^{(n)}\subseteq \ran C_{X}^{(n)}$ and therefore $h_g^{(n)}\defeq C_{X}^{(n)\dagger}C_{XY}^{(n)}g\in\cH$ is well defined for each $g\in\cG$.
	For each $y\in\cY$ and $h\in\cH$,
	\begin{equation}
		\label{equ:CMEmuNEWProjectionLimit}
		\innerprod{ h }{ \mu_{Y|X = \quark}(y) }_{ L^{2}(\bP_X)} = \lim_{n\to\infty} \innerprod{ h }{ \mu^{(n)}(\quark,y) }_{ L^{2}(\bP_X)},
	\end{equation}
	where, for $x\in\cX$ and $y\in\cY$,
	\[
		\mu^{(n)}(x,y) \defeq \big(\mu_{Y} + (C_{X}^{(n)\dagger} C_{XY}^{(n)})^{\ast} \, (\varphi(x) - \mu_{X})\big)(y).
	\]
	Suppose in addition that any of the following four conditions holds:
	\begin{enumerate}[label=(\roman*)]
		\item \label{enum:kCharacteristic} the kernel $k$ is characteristic;
		\item \label{enum:HCdense} $\cH_{\cC}$ is dense in $ L_{\cC}^{2}(\bP_X)$;
		\item \label{enum:Bstar} \Cref{assump:weakerCMElimit} holds;
		\item \label{enum:weakerBstar} $f_{\psi(y)}\in\overline{\cH_{\cC}}^{ L_{\cC}^{2}}$ for each $y\in\cY$.
	\end{enumerate}	
	Then, as $n\to\infty$,
	\begin{equation}
	\label{equ:CMElimit}
	\bignorm{ \mu^{(n)}(X,\quark) - \mu_{Y|X} }_{ L^{2}(\bP ; \cG)}	
	\to 0,
	\qquad	
	\bignorm{\mu_{Y|X = x} - \mu^{(n)}(x,\quark)}_{\cG}
	\to
	0
	\text{ for $\bP_X$-a.e.\ $x\in\cX$.}
	\end{equation}
\end{theorem}


\begin{proof}
	Note that, since $C$ is a trace-class operator, so is $C^{(n)}$.
	Furthermore, by \citet[Theorem~1]{baker1973joint}, $C_{XY}^{(n)} = (C_{X}^{(n)})^{1/2} V C_{Y}^{1/2}$ for some bounded operator $V\colon \cG\to\cH$.
	Since $C_{X}^{(n)}$ has finite rank, this implies that $\ran C_{XY}^{(n)}\subseteq \ran C_{X}^{(n)}$.
	Similarly to the proof of \Cref{thm:CMEproperly}, we define $c_g^{(n)} \defeq \bE[(f_g - h_g^{(n)})(X)]$ for $g\in\cG,n\in\bN$ and obtain by \eqref{equ:f_gExpectation} and \Cref{lemma:TeproducingPropertyAdjoint} for $x\in\cX$, $y\in\cY$ and $n\in\bN$ that
	\begin{equation}
		\label{equ:RewritingHlimit}
		h_{\psi(y)}^{(n)}(x) + c_{\psi(y)}^{(n)}
		=
		\mu^{(n)}(x,y).
	\end{equation}
	Identity \eqref{equ:CMEmuNEWProjectionLimit} can be obtained similarly to \eqref{equ:CMEmuNEWProjection} except that we also need to show that $\Cov[h(X),f_g(X)] = \lim_{n\to\infty} \Cov[h(X),h_g^{(n)}(X)]$ for all $h\in\cH$, as proved in \Cref{lemma:TechnicalDetailsMainProof}\ref{item:TechnicalDetailsMainProofA}.

	To establish \eqref{equ:CMElimit}, we first note that, by \Cref{lemma:TechnicalDetailsMainProof}\ref{item:TechnicalDetailsMainProofB}, for all $g \in \cG$, $[h_g^{(n)}]$ is the $ L_{\cC}^2$-orthogonal projection of $[f_g]$ onto $\cH_\cC^{(n)}$.
	Now let $y\in\cY$ and $U \defeq \bigcup_{n\in\bN} \cH_\cC^{(n)}$.
	Note that, by \Cref{lemma:EquivalenceCharacteristicDense}, \ref{enum:kCharacteristic}$\implies$\ref{enum:HCdense}$\implies$\ref{enum:Bstar}$\implies$\ref{enum:weakerBstar}, so let us assume \ref{enum:weakerBstar}.
	Since $\overline{U}^{\cH_\cC} = \cH_\cC$ and $[f_{\psi(y)}]\in\overline{\cH_\cC}^{ L_{\cC}^{2}}$ by assumption, and since \eqref{equ:HnormStrongerThanL2norm} implies that $\norm{ \quark }_\cH$ is a stronger norm than $\norm{ \quark }_{ L^{2}}$, we also have $[f_{\psi(y)}]\in\overline{U}^{ L_{\cC}^{2}}$ and \Cref{lemma:ConvergenceOfProjections} implies
	\begin{equation}
	\label{equ:ConvergenceL2_C}
	\bignorm{ [h_{\psi(y)}^{(n)}] - [f_{\psi(y)}] }_{ L_{\cC}^{2}}
	\xrightarrow[n\to\infty]{}
	0.
	\end{equation}	
	For $x\in\cX$ and $n\in\bN$ let $\fm^{(n)}(x) \defeq h_{\psi(\quark)}^{(n)}(x) = (C_{X}^{(n)\dagger} C_{XY}^{(n)})^{\ast}\varphi(x) \in\cG$ and $\fm(x) \defeq f_{\psi(\quark)}(x) = \mu_{Y|X=x} \in \cG$.
	Then $\fm^{(n)},\fm\in L^{2}(\bP_{X};\cG)$ by \eqref{equ:FiniteSecondMomentOfUandV}, since
	\begin{align*}
		\norm{\fm^{(n)}}_{ L^{2}(\bP_{X};\cG)}^{2}
		&=
		\bE\bigl[\bignorm{(C_{X}^{(n)\dagger} C_{XY}^{(n)})^{\ast}\varphi(X)}_{\cG}^{2}\bigr] 		
		\leq
		\bignorm{(C_{X}^{(n)\dagger} C_{XY}^{(n)})^{\ast}}\, \bE\bigl[\norm{\varphi(X)}_{\cH}^{2}\bigr] 	
		<
		\infty,
		\\
		\norm{\fm}_{ L^{2}(\bP_{X};\cG)}^{2}
		&=
		\bE\big[\norm{\bE[\psi(Y)|X]}_{\cG}^{2}\big] 		
		\leq
		\bE\big[\bE[\norm{\psi(Y)}_{\cG}^{2}|X]\big] 	
		=
		\bE\big[\norm{\psi(Y)}_{\cG}^{2}\big]
		<
		\infty.
	\end{align*}	
	So far, we have shown that, for each $y\in\cY$,
	\begin{itemize}
		\item		
		$([\fm^{(n)}](\quark))(y)$ is the $ L_{\cC}^{2}(\bP_{X})$-orthogonal projection of $([\fm](\quark))(y)$ onto $\cH_{\cC}^{(n)}$;
		\item		
		$([\fm^{(n)}](\quark))(y) \to ([\fm](\quark))(y)$ in $ L_{\cC}^{2}(\bP_{X})$ as $n \to \infty$.
	\end{itemize}
	Hence, by \Cref{lemma:LiftingPropertiesFromL2RtoL2G}\ref{item:LiftingPropertyA} and \ref{item:LiftingPropertyB},
	\begin{equation}
	\label{equ:ConvergenceL2_P_X_G}
	\bignorm{ (\fm^{(n)}(X) - \bE[\fm^{(n)}(X)]) - (\fm(X) - \bE[\fm(X)]) }_{ L^{2}(\bP ; \cG)}
	\xrightarrow[n\to\infty]{} 0.
	\end{equation}
	Therefore, by \eqref{equ:RewritingHlimit} and the definition of $c_g^{(n)}$, $\mu^{(n)}(X,\quark)$ converges to $\mu_{Y|X} = f_{\psi(\quark)}(X) = \fm(X)$ in $ L^{p}(\bP ; \cG)$ for $p=2$ and, since $\bP$ is a finite measure, also for $p=1$.
	By \Cref{lemma:MartingaleConsequences}, $(\mu^{(n)}(X,\cdot))_{n\in\bN}$ is a martingale, and so \citet[Theorem V.2.8]{diestel1977} implies that this convergence even holds a.e., i.e.\ $\mu^{(n)}(x,\quark)$ converges in $\cG$ to $\mu_{Y|X=x}$ for $\bP_{X}$-a.e.\ $x\in\cX$.
\end{proof}

%% file: section-05-uncentred.tex
\section{Theory for Uncentred Operators}
\label{section:UncentredOperators}

Beginning with the work of \cite{song2010a,song2010b}, uncentred \mbox{(cross-)}\-covariance operators became more commonly used than centred ones. This section shows how similar results to those of \Cref{section:TheoryCentred} can be obtained for uncentred operators.
Roughly speaking, the same conclusions can be made as in \Cref{thm:CMEproperly} but under \Cref{assump:strongCME} in place of \ref{assump:weakerCME}, while only weaker statements than in \Cref{thm:CMEproperlyLimit} can be obtained in \Cref{thm:CMEproperlyUncenteredLimit} (no $\bP_{X}$-a.e.\ convergence, see below) and again under the stronger \Cref{assump:strongCMElimit} in place of \ref{assump:weakerCMElimit}.
This observation suggests that centred operators are superior to uncentred ones in terms of generality.
So far, the theoretical justification for CME using uncentred operators relies on \citet[Theorems 1 and 2]{fukumizu2013kernel}, which require rather strong assumptions.
Our improvement can be summarised as follows:
\begin{itemize}
	\item Since we use $\uu{C}_{X}^{\dagger}$ instead of $\uu{C}_{X}^{-1}$ our theory can cope with non-injective operators $\uu{C}_{X}$.
	This is only a minor advance, since $\uu{C}_{X}$ is injective under rather mild conditions on $X$ and $k$ (see \citet[Footnote~3]{fukumizu2013kernel}).
	\item In contrast to \citet[Theorem~2]{fukumizu2013kernel}, we do not require the assumption that $\varphi(x)$ lies in the range of $\uu{C}_{X}$.
	The reason for this is that the operator $(\uu{C}_{X}^{\dagger} \uu{C}_{XY})^{\ast}$ in \eqref{equ:CMEmuNEW} is globally defined whereas $\uu{C}_{YX}\uu{C}_{X}^{-1}$ is not.
	This is an important improvement since the assumption that $\varphi(x) \in \ran \uu{C}_{X}$ is typically hard to verify (see \Cref{footnote:Picard}).
	\item We state a version of the CME formula under \Cref{assump:strongCMElimit}, which is a verifiable condition since it follows from the kernel $k$ being $L^2$-universal.
	\item As explained in \Cref{remark:AllResultsOnlyP_Xae}, the condition in \citet[Theorem~2]{fukumizu2013kernel} on $\bE[g(Y)|X=\quark]$ to lie in $\cH$ for each $g\in\cG$ is ill posed, since these functions are uniquely defined only $\bP_X$-a.e.
	However, in our case, \Cref{assumption:CME}\ref{notation:NoNontrivialZerosInH} ensures that Assumptions \ref{assump:strongCME} and \ref{assump:strongCMElimit} are unambiguous.
\end{itemize}
As mentioned above, using centred operators instead of uncentred ones yields the important advantage of requiring only the weaker \Cref{assump:weakerCME} in place of \ref{assump:strongCME} or \Cref{assump:weakerCMElimit} in place of \ref{assump:strongCMElimit}, respectively.
Further, \Cref{thm:CMEproperlyUncenteredLimit} provides weaker statements than its centred analogue, \Cref{thm:CMEproperlyLimit}:
we show only convergence in $L^2(\bP_{X};\cG)$, which does not guarantee convergence for $\bP_{X}$-a.e.\ $x\in\cX$.

\begin{theorem}
	\label{thm:RelationC_XandC_XYuncentred}
	Under \Cref{assumption:CME}, the following statements are equivalent:
	\begin{enumerate}[label=(\roman*)]
		\item \label{statement1lemma:RelationC_XandC_XYuncentred} \Cref{assump:weakCMEuncentred} holds.
		\item \label{statement2lemma:RelationC_XandC_XYuncentred} For each $g\in\cG$ there exists $h_g\in\cH$ such that
		$
		\uu{C}_{X} h_g
		=
		\uu{C}_{XY}g.
		$
		\item \label{statement3lemma:RelationC_XandC_XY_uncentred} $\ran \uu{C}_{XY}\subseteq \ran \uu{C}_{X}$.
	\end{enumerate}
\end{theorem}


\begin{proof}
	The proof is identical to that of \Cref{thm:RelationC_XandC_XY} (apart from using uncentred covariance operators in place of centred ones).
\end{proof}


Similar to \Cref{prop:SpecificChoiceh_g}, the element $h_g\in\cH$ in \Cref{assump:weakCMEuncentred} can always be chosen as $h_g = \uu{C}_{X}^{\dagger} \uu{C}_{XY}g$.


\begin{proposition}
	\label{prop:SpecificChoiceh_gUncentred}
	Let \Cref{assumption:CME} hold.
	Under \Cref{assump:weakCMEuncentred}, $h_g$ may be chosen as
	\begin{equation}
		\label{equ:SpecificChoiceh_gUncentred}
		h_g = \uu{C}_{X}^{\dagger} \uu{C}_{XY}g.
	\end{equation}
	More precisely, $\uu\Cov[(\uu{C}_{X}^{\dagger} \uu{C}_{XY}g)(X)-f_g(X),h(X)] = 0$ for all $h\in\cH$ and $g\in\cG$.
	If $f_g\in\cH$ for some $g\in\cG$, then the identity $f_g = \uu{C}_{X}^{\dagger} \uu{C}_{XY}g$ holds $\bP_X$-a.e.	
\end{proposition}


\begin{proof}
	By \Cref{thm:RelationC_XandC_XYuncentred}, \eqref{equ:SpecificChoiceh_gUncentred} is well defined.
	If \Cref{assump:weakCMEuncentred} holds, then, by \Cref{thm:RelationC_XandC_XYuncentred} and \Cref{lemma:RelationCovAndC_XY}, for all $g\in\cG$ and $h\in\cH$,
	\begin{align*}
		\uu{\Cov} [h(X), (\uu{C}_{X}^{\dagger} \uu{C}_{XY}g)(X)]
		& =
		\innerprod{ h }{ \uu{C}_{X} \uu{C}_{X}^{\dagger} \uu{C}_{XY} g }_{\cH} \\
		& =
		\innerprod{ h }{ \uu{C}_{XY} g }_{\cH}
		=
		\uu{\Cov} [h(X), f_g(X)].
	\end{align*}
	If $f_g\in\cH$ holds for some $g\in\cG$, then \Cref{lemma:RelationCovAndC_XY} implies that $\uu{C}_{X}f_g = \uu{C}_{XY} g$ for all $g\in\cG$ and the claim follows from \Cref{lemma:kernelC_Xconstants}.
\end{proof}


Let us now formulate and prove the analogues of \Cref{thm:CMEproperly,thm:CMEproperlyLimit} for uncentred operators.

\begin{theorem}[Uncentred CME]
	\label{thm:CMEproperlyUncentered}
	Under Assumptions \ref{assumption:CME} and \ref{assump:weakCMEuncentred}, the linear operator\linebreak[4] $\uu{C}_{X}^{\dagger} \uu{C}_{XY}\colon \cG\to\cH$ is bounded\textsuperscript{\ref{footnote:Carleman_trick}} and, for all $y\in\cY$ and $h\in\cH$,
	\begin{align}
		\label{equ:CMEmuNEWProjectionAlternativeUncentred}
		\innerprod{ h }{ \mu_{Y|X = \quark}(y) }_{ L^2(\bP_X)}
		&=
		\Innerprod{ h }{ \big((\uu{C}_{X}^{\dagger} \uu{C}_{XY})^{\ast} \varphi(\quark)\big)(y) }_{ L^2(\bP_X)}.
		\intertext{
		Suppose in addition that any of the following four conditions holds:
		\begin{enumerate}[label=(\roman*)]
			\item the kernel $k$ is $L^2$-universal;
			\item $\cH$ is dense in $ L^2(\bP_X)$;
			\item \Cref{assump:strongCME} holds;
			\item $f_{\psi(y)}\in\cH$ for each $y\in\cY$.
		\end{enumerate}
		Then, for $\bP_X$-a.e.\ $x\in\cX$,
		}
		\label{equ:CMEmuNEWAlternativeUncentred}
		\mu_{Y|X = x}
		&=
		(\uu{C}_{X}^{\dagger} \uu{C}_{XY})^{\ast} \, \varphi(x).
	\end{align}
\end{theorem}


\begin{proof}
	First note that, by \Cref{thm:RelationC_XandC_XYuncentred,theorem:DouglasExistenceQ}, $\uu{C}_{X}^{\dagger} \uu{C}_{XY}$ is well defined and bounded\textsuperscript{\ref{footnote:Carleman_trick}} and that for each $g\in\cG$ we may choose the function $h_g\in\cH$ in \Cref{assump:weakCMEuncentred} as $h_g = \uu{C}_{X}^{\dagger} \uu{C}_{XY}g$ by \Cref{prop:SpecificChoiceh_gUncentred}.
	By \Cref{lemma:TeproducingPropertyAdjoint} we obtain, for all $x\in\cX$ and $y\in\cY$,
	\begin{align*}
		h_{\psi(y)}(x)
		&=
		\big( (\uu{C}_{X}^{\dagger} \uu{C}_{XY})^{\ast} \varphi(x)\big)(y).
	\end{align*}
	This yields \eqref{equ:CMEmuNEWProjectionAlternativeUncentred} via
	\begin{align*}
		\Innerprod{ h }{ \big(\mu_{Y|X = \quark} - (\uu{C}_{X}^{\dagger} \uu{C}_{XY})^{\ast} \, \varphi(\quark)\big)(y) }_{L^2(\bP_X)}
		&=
		\innerprod{ h }{ f_{\psi(y)} - h_{\psi(y)} }_{L^2(\bP_X)} \\
		& =
		\uu{\Cov}[h(X),(f_{\psi(y)}-h_{\psi(y)})(X)]
		=
		0,
	\end{align*}
	which implies \eqref{equ:CMEmuNEWAlternativeUncentred} under any of the four conditions stated in the theorem (possibly using \Cref{prop:SpecificChoiceh_gUncentred}).
\end{proof}



\begin{theorem}[Uncentred CME under finite-rank approximation]
	\label{thm:CMEproperlyUncenteredLimit}
	Let \Cref{assumption:CME} hold.
	Further, let $(h_n)_{n\in\bN}$ be a complete orthonormal system of $\cH$ that is an eigenbasis of $C_{X}$, let $\cH^{(n)} \defeq \spn \{ h_1,\dots,h_n \}$, let $\cF \defeq \cG\oplus\cH$, let $P^{(n)}\colon\cF\to\cF$ be the orthogonal projection onto $\cG\oplus \cH^{(n)}$, and let
	\[
	\uu{C}
	\defeq
	\begin{pmatrix}
	\uu{C}_{Y} & \uu{C}_{YX}
	\\
	\uu{C}_{XY} & \uu{C}_{X}
	\end{pmatrix},
	\qquad
	\uu{C}^{(n)}
	\defeq
	P^{(n)}\uu{C}P^{(n)}
	=
	\begin{pmatrix}
	\uu{C}_{Y} & \uu{C}_{YX}^{(n)}
	\\
	\uu{C}_{XY}^{(n)} & \uu{C}_{X}^{(n)}
	\end{pmatrix}.
	\]
	Then $\ran \uu{C}_{XY}^{(n)}\subseteq \ran \uu{C}_{X}^{(n)}$ and therefore $\uu{h}_g^{(n)}\defeq \uu{C}_{X}^{(n)\dagger}\uu{C}_{XY}^{(n)}g\in\cH$ is well defined for each $g\in\cG$.
	For each $y\in\cY$ and $h\in\cH$,
	\begin{equation}
	\label{equ:CMEmuNEWProjectionLimitUncentred}
	\innerprod{ h }{ \mu_{Y|X = \quark}(y) }_{ L^{2}(\bP_X)} = \lim_{n\to\infty} \innerprod{ h }{ \mu^{(n)}(\quark,y) }_{ L^{2}(\bP_X)},
	\end{equation}
	where, for $x\in\cX$ and $y\in\cY$,
	\begin{equation}
	\label{equ:DefinitionU_Mu_n}
	\uu{\mu}^{(n)}(x,y) \defeq \big((\uu{C}_{X}^{(n)\dagger} \uu{C}_{XY}^{(n)})^{\ast} \varphi(x)\big)(y).
	\end{equation}
	Suppose in addition that any of the following four conditions holds:
	\begin{enumerate}[label=(\roman*)]
		\item \label{enum:kUniversal} the kernel $k$ is $L^{2}$-universal;
		\item \label{enum:Hdense} $\cH$ is dense in $L^{2}(\bP_X)$;
		\item \label{enum:Astar} \Cref{assump:strongCMElimit} holds;
		\item \label{enum:weakerAstar} $f_{\psi(y)}\in\overline{\cH}^{ L^{2}}$ for each $y\in\cY$.
	\end{enumerate}
	Then
	\begin{equation}
	\label{equ:CMElimitUncentred}
	\bignorm{ \uu{\mu}^{(n)}(X,\quark) - \mu_{Y|X} }_{ L^{2}(\bP ; \cG)}	
	\xrightarrow[n\to\infty]{} 0.
	\end{equation}
\end{theorem}


\begin{proof}
	The proof goes analogously to the one of \Cref{thm:CMEproperlyLimit} up to equation	\eqref{equ:ConvergenceL2_P_X_G}, using uncentred operators instead of centred ones and the statements \ref{item:LiftingPropertyC}, \ref{item:LiftingPropertyD} instead of \ref{item:LiftingPropertyA}, \ref{item:LiftingPropertyB} of \Cref{lemma:TechnicalDetailsMainProof,lemma:LiftingPropertiesFromL2RtoL2G}.
	However, we cannot draw the final conclusion of convergence almost everywhere since we do not have the martingale property, which is provided by \Cref{lemma:MartingaleConsequences} for the centred case.	
	Note that our proof relies on \citet[Theorem~1]{baker1973joint} which, strictly speaking, only treats the centred case, but its uncentred version can be proven similarly.
\end{proof}

\begin{corollary}
	Under the assumptions of \Cref{thm:CMEproperlyUncentered} (including either of the additional ones),
	\[
	\mu_{Y} = (\uu{C}_{X}^{\dagger} \uu{C}_{XY})^{\ast} \mu_{X}.
	\]
	Under the assumptions of \Cref{thm:CMEproperlyUncenteredLimit} (including either of the additional ones),
	\[
	\bignorm{\mu_{Y} - (\uu{C}_{X}^{(n)\dagger} \uu{C}_{XY}^{(n)})^{\ast} \mu_{X}}_{\cG}
	\xrightarrow[n\to\infty]{} 0.
	\]
\end{corollary}


\begin{proof}
	As stated in \Cref{thm:CMEproperlyUncentered}, $\uu{C}_{X}^{\dagger} \uu{C}_{XY}$ is a well-defined and bounded\textsuperscript{\ref{footnote:Carleman_trick}} linear operator.
	Hence, by the law of total expectation and \Cref{thm:CMEproperlyUncentered},
	\[
	\mu_{Y}
	=
	\bE[\mu_{Y|X}]
	=
	\bE\big[ (\uu{C}_{X}^{\dagger} \uu{C}_{XY})^{\ast} \varphi(X) \big]
	=
	(\uu{C}_{X}^{\dagger} \uu{C}_{XY})^{\ast} \bE[\varphi(X)]
	=
	(\uu{C}_{X}^{\dagger} \uu{C}_{XY})^{\ast} \mu_{X} ,
	\]
	proving the first claim. The second one follows from Jensen's inequality and \Cref{thm:CMEproperlyUncenteredLimit} via
	\begin{align*}
	\bignorm{\mu_{Y} - (\uu{C}_{X}^{(n)\dagger} \uu{C}_{XY}^{(n)})^{\ast} \mu_{X}}_{\cG}^{2}
	&=
	\bignorm{\bE[\mu_{Y|X}] - (\uu{C}_{X}^{(n)\dagger} \uu{C}_{XY}^{(n)})^{\ast} \, \bE[\varphi(X)]}_{\cG}^{2}
	\\
	&=
	\bignorm{\bE \big[\mu_{Y|X} - (\uu{C}_{X}^{(n)\dagger} \uu{C}_{XY}^{(n)})^{\ast} \varphi(X)\big]}_{\cG}^{2}
	\\
	& \le
	\bignorm{\mu_{Y|X} - (\uu{C}_{X}^{(n)\dagger} \uu{C}_{XY}^{(n)})^{\ast} \varphi(X)}_{L^2(\bP_{X} ; \cG)}^{2}
	\xrightarrow[n\to\infty]{} 0.
	\end{align*}
\end{proof}


%
%

%% file: section-06-Gaussian_conditioning.tex

\section{Gaussian Conditioning in Hilbert spaces}
\label{section:GaussianConditioning}

This section gives a review of conditioning theory for Gaussian random variables in separable Hilbert spaces, summarising the work of \citet{owhadi2015conditioning}.
Our only somewhat novel contribution here is the explicit characterisation of the essential operator $\widehat Q_{C,\cH}$ in terms of the Moore--Penrose pseudo-inverse, which appears as an exercise for the reader in \citet[Remark~2.3]{arias2008gi_douglas}.

In the following let $\cF = \cG\oplus\cH$ be the sum of two separable Hilbert spaces $\cG$ and $\cH$ and let $(U, V)$ be an $\cF$-valued jointly Gaussian random variable with mean $\mu\in\cF$ and covariance operator $C\colon\cF\to\cF$ given by the following block structures:
\[
	\begin{pmatrix}
	U\\ V
	\end{pmatrix}
	\sim\cN(\mu,C),
	\qquad
	\mu
	=
	\begin{pmatrix}
	\mu_U
	\\
	\mu_V
	\end{pmatrix},
	\qquad
	C
	=
	\begin{pmatrix}
	C_U & C_{UV}
	\\
	C_{VU} & C_V
	\end{pmatrix}\ge 0
\]
with $\mu_{U} \in \cG$, etc.
We denote by $L(\cF)$ the Banach algebra of bounded linear operators on $\cF$ and by $L_+(\cF) = \{A\in L(\cF) \mid A\ge 0 \}$ the set of positive operators, i.e.\ those self-adjoint operators $A$ for which $\innerprod{ x }{ A x } \geq 0$ for all $x\in\cF$.
The theory of Gaussian conditioning relies on the concept of so-called \emph{oblique projections}:

\begin{definition}
	\label{def:ObliqueProjections}
	Let $\cF = \cG\oplus \cH$ be a direct sum of two Hilbert spaces $\cG$ and $\cH$ and let $C\in L_+(\cF)$ be a positive operator.
	The set of ($C$-symmetric) \emph{oblique projections} onto $\cH$ is given by
	\[
		\cP(C,\cH)
		=
		\{Q\in L(\cF) \mid Q^2=Q,\ \ran Q = \cH,\ CQ = Q^{\ast} C\}.
	\]
	The pair $(C,\cH)$ is said to be \emph{compatible} if $\cP (C,\cH)$ is non-empty.
\end{definition}

The first two conditions $Q^2=Q$ and $\ran Q = \cH$ imply that $Q$ has the block structure
\begin{equation}
	\label{equ:BlockStructureQ}
	Q=\begin{pmatrix}
	0 & 0 \\ \widehat{Q} & \Id_{\cH}
	\end{pmatrix},
	\qquad
	\widehat{Q} \colon \cG\to\cH.
\end{equation}
Then, the condition $CQ = Q^{\ast}C$ is equivalent to $C_V \widehat{Q} = C_{VU}$ (which follows from a straightforward blockwise multiplication, see \Cref{lemma:SufficientConditionOblique}) and implies in particular $\ran C_{VU}\subseteq\ran C_V$.
The other way round, as we will see later on, the condition $\ran C_{VU}\subseteq\ran C_V$ guarantees the existence of an oblique projection $Q\in \cP(C,\cH)$ and will provide a crucial link between the theory of Gaussian conditioning and conditional mean embeddings in \Cref{section:ConnectionGaussianCME}.

The results on conditioning Gaussian measures can then be summarised as follows:


\begin{theorem}[{\citealp[Theorem~3.3, Corollary~3.4]{owhadi2015conditioning}}]
	\label{thm:MainTheoremOwhadiScovel}
	If $(C,\cH)$ is compatible, then conditioning $U$ on $V=v\in\cH$ results in a Gaussian random variable on $\cG$ with mean $\mu_{U|V = v}$ and covariance operator $C_{U|V = v}$ given by
	\begin{equation}
		\label{equ:GaussianConditioningCompatibleCase}
		\begin{cases}
		\mu_{U|V = v} = \mu_U + \widehat{Q}^{\ast}(v-\mu_V),
		\\
		C_{U|V = v} = C_U - C_{UV} \widehat{Q}
		\end{cases}
	\end{equation}
	for any oblique projection $Q\in \cP(C,\cH)$ given in the form \eqref{equ:BlockStructureQ}.
	Also, in this case, $\cP(C,\cH)$ contains a unique element
	\[
		Q_{C,\cH} =\begin{pmatrix}
		0 & 0 \\ \widehat{Q}_{C,\cH} & \Id_{\cH}
		\end{pmatrix}
	\]
	that fulfils the properties \eqref{equ:SufficientConditionsUniqueOblique} defined below.

	If $(C,\cH)$ is incompatible, then conditioning $U$ on $V=v\in\cH$ still yields a Gaussian random variable on $\cG$, but the corresponding formulae for the conditional mean $\mu_{U|V = v}$ and covariance operator $C_{U|V = v}$ are given by a limiting process using finite-rank approximations of $C$ in the following way.
	Let $(h_n)_{n\in\bN}$ be a complete orthonormal system of $\cH$, $P^{(n)}\colon\cF\to\cF$ denote the orthogonal projection onto $\cG\oplus \spn \{ h_1,\dots,h_n \}$ and $C^{(n)} = P^{(n)}CP^{(n)}$.
	Then $(C^{(n)},\cH)$ is compatible for each $n\in\bN$ and, for $\bP_{V}$-a.e.\ $v\in\cH$ (with $\bP_{V}$ denoting the distribution of $V$),
	\begin{equation}
		\label{equ:GaussianConditioningIncompatibleCase}
		\begin{cases}
		\mu_{U|V = v} = \mu_U + \lim_{n\to\infty}\widehat{Q}_{C^{(n)},\cH}^{\ast}(v-\mu_V),
		\\[0.1cm]
		C_{U|V = v} = C_U - \lim_{n\to\infty} C_{UV} \widehat{Q}_{C^{(n)},\cH},
		\end{cases}
	\end{equation}
	where the second limit is in the trace norm.
\end{theorem}


In the following we will revisit some theory on oblique projections which will be necessary to establish the connection between Gaussian conditioning and conditional mean embeddings.
We will also characterise the special oblique projection $Q_{C,\cH}\in \cP(C,\cH)$ by means of the Moore--Penrose pseudo-inverse.

\begin{lemma}
	\label{lemma:SufficientConditionOblique}
	If $\widehat{Q}\colon \cG\to\cH$ is a bounded linear operator such that $C_V\widehat{Q} = C_{VU}$, then
	\[
		Q = \begin{pmatrix} 0 & 0\\ \widehat{Q} & \Id_{\cH} \end{pmatrix}\in \cP(C,\cH).
	\]
	In particular, the pair $(C,\cH)$ is compatible.
\end{lemma}

\begin{proof}
	The properties $Q^2=Q$ and $\ran Q = \cH$ are clear from the definition of $Q$ and a straightforward blockwise multiplication shows that $CQ = Q^{\ast} C$.
\end{proof}

\begin{proposition}
	\label{prop:CompatibleImpliesSpecialElement}
	In the setup of \Cref{def:ObliqueProjections}, if $(C,\cH)$ is compatible, then there exists a unique bounded operator $\widehat{Q}_{C,\cH}\colon \cG\to\cH$ such that
	\begin{equation}
	\label{equ:SufficientConditionsUniqueOblique}
	C_V\widehat{Q}_{C,\cH} = C_{VU},
	\qquad
	\ker \widehat{Q}_{C,\cH} = \ker C_{VU},
	\qquad
	\ran \widehat{Q}_{C,\cH} \subseteq \overline{\ran C_V}.
	\end{equation}
	By \Cref{lemma:SufficientConditionOblique} the first property implies that
	\[
	Q_{C,\cH} = \begin{pmatrix} 0 & 0\\ \widehat{Q}_{C,\cH} & \Id_{\cH} \end{pmatrix}\in \cP(C,\cH).
	\]
\end{proposition}


\begin{proof}
	See \citet[Theorem~1]{douglas1966majorization} or \citet[Theorem~2.1]{fillmore1971operator} for the existence and uniqueness of $\widehat{Q}_{C,\cH}$ and \citet{corach2000oblique} or \citet{owhadi2015conditioning} for its connection to oblique projections.
\end{proof}

If one follows the original construction of \citet[Theorem~1]{douglas1966majorization} or \citet[Theorem~2.1]{fillmore1971operator}, it is easy to see how this unique element can be characterised in terms of the Moore--Penrose pseudo-inverse $C_V^\dagger$ of $C_V$:

\begin{theorem}
	\label{thm:UniqueProjectionPseudoInverse}
	If $\ran C_{VU} \subseteq \ran C_V$, then $\widehat{Q} = C_V^\dagger C_{VU}\colon\cG\to\cH$ is a well-defined bounded operator which uniquely fulfils the conditions \eqref{equ:SufficientConditionsUniqueOblique}.
\end{theorem}


\begin{proof}
	This is a direct application of \Cref{theorem:DouglasExistenceQ}.
\end{proof}


\begin{theorem}
	\label{thm:UniqueObliqueProjectionEqualsPseudoInverse}
	In the setup of \Cref{def:ObliqueProjections}, the following statements are equivalent:
	\begin{enumerate}[label=(\roman*)]
		\item \label{item:UniqueObliqueProjectionEqualsPseudoInverse_1} $(C,\cH)$ is compatible.
		\item \label{item:UniqueObliqueProjectionEqualsPseudoInverse_2} $\ran C_{VU} \subseteq \ran C_V$.
	\end{enumerate}
	If either of these conditions holds, then the unique element $Q_{C,\cH}\in \cP(C,\cH)$ in \Cref{prop:CompatibleImpliesSpecialElement} is given by
	\begin{equation}
	\label{UniqueObliqueEqualsPseudoInverse}
	\widehat{Q}_{C,\cH}
	=
	C_V^\dagger C_{VU}.
	\end{equation}
\end{theorem}


\begin{proof}
	If $(C,\cH)$ is compatible, there exists an element $\widehat{Q}_{C,\cH}\colon \cG\to\cH$ with $C_V\widehat{Q}_{C,\cH} = C_{VU}$ by \Cref{prop:CompatibleImpliesSpecialElement}, which implies \ref{item:UniqueObliqueProjectionEqualsPseudoInverse_2}.
	If $\ran C_{VU} \subseteq \ran C_V$, then \Cref{thm:UniqueProjectionPseudoInverse} and \Cref{lemma:SufficientConditionOblique} imply \ref{item:UniqueObliqueProjectionEqualsPseudoInverse_1}.
	\Cref{thm:UniqueProjectionPseudoInverse} and the uniqueness of $\widehat{Q}_{C,\cH}$ in \Cref{prop:CompatibleImpliesSpecialElement} imply \eqref{UniqueObliqueEqualsPseudoInverse}.
\end{proof}

\begin{remark}
	\Cref{lemma:SufficientConditionOblique} and the equivalence part of \Cref{thm:UniqueObliqueProjectionEqualsPseudoInverse} were already proved by \citet{corach2000oblique};
	we state them for the sake of readability.
	The second part of \Cref{thm:UniqueObliqueProjectionEqualsPseudoInverse}\ref{item:UniqueObliqueProjectionEqualsPseudoInverse_2} characterises the operator $\widehat{Q}_{C,\cH}$ in terms of the Moore--Penrose pseudo-inverse, without an assumption of closed range, as anticipated by \citet[Remark~2.3]{arias2008gi_douglas}.
\end{remark}

There are covariance operators $C$ for which the above conditions do not hold:

\begin{example}
	Let $\cH=\cG$ be any (separable) infinite-dimensional Hilbert space with complete orthonormal basis $(e_j)_{j\in\bN}$.
	Let
	\begin{align*}
		C_U & \defeq \sum_{j\in\bN} j^{-2} e_j\otimes e_j, \\
		C_V & \defeq \sum_{j\in\bN} j^{-4} e_j\otimes e_j, \\
		C_{VU} & = C_{UV} \defeq C_U^{1/2} \Id_{\cH} C_V^{1/2} = \sum_{j\in\bN} j^{-3} e_j\otimes e_j.
	\end{align*}
	By \citet[Theorem~2]{baker1973joint},
	\[
		C \defeq \begin{pmatrix} C_U & C_{UV}\\ C_{VU} & C_V \end{pmatrix}
	\]
	is a legitimate positive definite covariance operator on $\cF = \cG\oplus\cH$.
	However,
	\[
		\ran C_{VU}
		=
		\left\{ \sum_{j\in\bN}\alpha_j e_j \,\middle|\, (j^3\alpha_j)_{j\in\bN}\in \ell^2 \right\}
		\not\subseteq
		\left\{ \sum_{j\in\bN}\alpha_j e_j \,\middle|\, (j^4\alpha_j)_{j\in\bN}\in \ell^2 \right\}
		=
		\ran C_V.
	\]
\end{example}

%% file: section-07-connection_Gaussian_CME.tex

\section{Connection between CME and Gaussian Conditioning}
\label{section:ConnectionGaussianCME}

If we compare the theories of CMEs and Gaussian conditioning in Hilbert spaces, we make the following observations:
\begin{itemize}
	\item Formula \eqref{equ:CMEmuNEW} for CME and formula \eqref{equ:GaussianConditioningCompatibleCase} for Gaussian conditioning look very similar (in view of \Cref{thm:UniqueObliqueProjectionEqualsPseudoInverse}).
	\item The assumptions under which the conditioning process is ``easy'' --- namely \Cref{assump:weakCME} (as long as \Cref{assump:weakerCMElimit} holds as well) and the compatibility of $(C,\cH)$ --- are equivalent to the conditions that $\ran C_{XY} \subseteq \ran C_{X}$ and $\ran C_{VU} \subseteq \ran C_V$ respectively (\Cref{thm:RelationC_XandC_XY,thm:UniqueObliqueProjectionEqualsPseudoInverse}).
\end{itemize}
This motivates us to connect these two theories by working in the setup of \Cref{section:CMESetup} and introducing new jointly Gaussian random variables $U$ and $V$ that take values in the RKHSs $\cG$ and $\cH$ respectively,
where the means $\mu_U$ and $\mu_V$ and \mbox{(cross-)}\-covariance operators
$C_U, C_{UV}, C_{VU}$, and $C_V$ are chosen to coincide with the kernel mean embeddings $\mu_{Y}$ and $\mu_{X}$ and the kernel \mbox{(cross-)}\-covariance operators $C_{Y}, C_{YX}, C_{XY}$, and $C_{X}$ respectively:
\begin{align}
	\label{equ:ChoiceUV}
	\begin{pmatrix}
	U\\ V
	\end{pmatrix}
	& \sim\cN(\mu,C),
	&
	\mu
	& =
	\begin{pmatrix}
	\mu_U
	\\
	\mu_V
	\end{pmatrix}
	\defeq
	\begin{pmatrix}
	\mu_{Y}
	\\
	\mu_{X}
	\end{pmatrix},
	&
	C
	& =
	\begin{pmatrix}
	C_U & C_{UV}
	\\
	C_{VU} & C_V
	\end{pmatrix}
	\defeq
	\begin{pmatrix}
	C_{Y} & C_{YX}
	\\
	C_{XY} & C_{X}
	\end{pmatrix}.
\end{align}
By \citet[Theorem~1]{baker1973joint} and since \Cref{assumption:CME}\ref{notation:CMErvs} implies that $C$ is a trace-class covariance operator, the Gaussian random variable $(U,V)$ is well defined in $\cG \oplus \cH$.
Note that the random variables $W=(U,V)$ and $Z=(\psi(Y),\varphi(X))$ do not coincide even though they have the same mean and covariance operator, since the latter will not generally be Gaussian.
Surprisingly, their conditional means agree, as long as we condition on $V=v=\varphi(x)$ and $X=x$ respectively.
This is obvious when one compares \eqref{equ:CMEmuNEW} with \eqref{equ:GaussianConditioningCompatibleCase} (and \eqref{equ:CMElimit} with \eqref{equ:GaussianConditioningIncompatibleCase} using \Cref{thm:UniqueProjectionPseudoInverse}).
A natural question is whether a similar equality holds for the conditional covariance operator $C_{Y|X=x}$.
However, the covariance operator $C_{U|V=v}$ obtained from Gaussian conditioning is independent of $v$, a special property of Gaussian measures that cannot be expected of the conditional kernel covariance operator $C_{Y|X=x}$.
Instead, $C_{U|V=v}$ equals the mean of $C_{Y|X=x}$ when averaged over all possible outcomes $x\in\cX$.\footnote{This observation has already been made by \citet[Proposition 5]{fukumizu2004dimensionality} under stronger assumptions and by \citet[Proposition~3]{fukumizu2009kernel} in a weaker form.}
These insights are summarised in the following proposition and illustrated in \Cref{fig:CommutativeDiagramCMEGaussian}.

\begin{figure}[t]
	\centering
	\adjustbox{scale=0.9}{
		\begin{tikzcd}[column sep=3em,row sep=7em]
			\cX,\cY
			&
			\cH,\cG
			&
			\text{Gaussian on } \cG\oplus \cH
			\\[-2.3cm]
			\begin{cases}
				\begin{rcases}
					\qquad\ x\in\cX \\
					\qquad X\sim\bP_X \\
					\qquad \, Y\sim\bP_Y \\
					(X,Y)\sim\bP_{XY}
				\end{rcases}
			\end{cases}
			\arrow{r}[sloped,above]{\text{embed}}[sloped,below]{\psi,\varphi}
			\arrow{d}[sloped,above]{\text{conditioning on}}[sloped,below]{X=x}
			&
			\begin{cases}
				\begin{rcases}
					\varphi(x) \\
					\textcolor{blue}{\psi(Y), \varphi(X)} \\
					\mu_{Y}, C_{Y}, C_{YX} \\
					\mu_{X}, C_{XY}, C_{X} \\
				\end{rcases}
			\end{cases}
			\arrow[leftrightarrow]{r}
			\arrow{d}[sloped,above]{\text{conditional mean}}[sloped,below]{\text{embedding}}
			&
			\begin{pmatrix} \textcolor{blue}{U}\\ \textcolor{blue}{V}	\end{pmatrix}
			\sim \cN \left(
			\begin{pmatrix} \mu_{Y} \\ \mu_{X}	\end{pmatrix},
			\begin{pmatrix} \ C_{Y} \hspace{0.2cm} C_{YX}\\ C_{XY} \hspace{0.2cm} C_{X}\ 	\end{pmatrix}
			\right)
			\arrow{d}[sloped,above]{\text{conditioning on}}[sloped,below]{V=v = \varphi(x)}
			\\
			(Y|X=x) \sim \bP_{Y|X=x}
			\arrow{r}[sloped,above]{\text{embed}}[sloped,below]{\psi,\varphi}
			&
			\mu_{Y|X=x}, C_{Y|X=x}
			\arrow[leftrightarrow]{r}
			&
			(U|V=v)\sim \cN(\mu_{U|V = v}, C_{U|V = v})
		\end{tikzcd}
	}
	\caption{A normally-distributed $\cG\oplus\cH$-valued normal random variable $(U,V)$ can be defined with the same mean and covariance structure as $(\psi(Y),\varphi(X))$.
	While the latter will typically fail to be normally distributed, surprisingly, the conditional means of the two random variables happen to agree!
	Since $C_{U|V = v}$ does not depend on the realisation $v$, a specific property of Gaussian random variables that cannot be expected from $C_{Y|X=x}$, a similar agreement for the conditional covariance operators cannot be obtained.
	Instead, the identity provided by \Cref{thm:ConnectionToGaussianRV} holds, which is open to interpretation.
	}
	\label{fig:CommutativeDiagramCMEGaussian}
\end{figure}
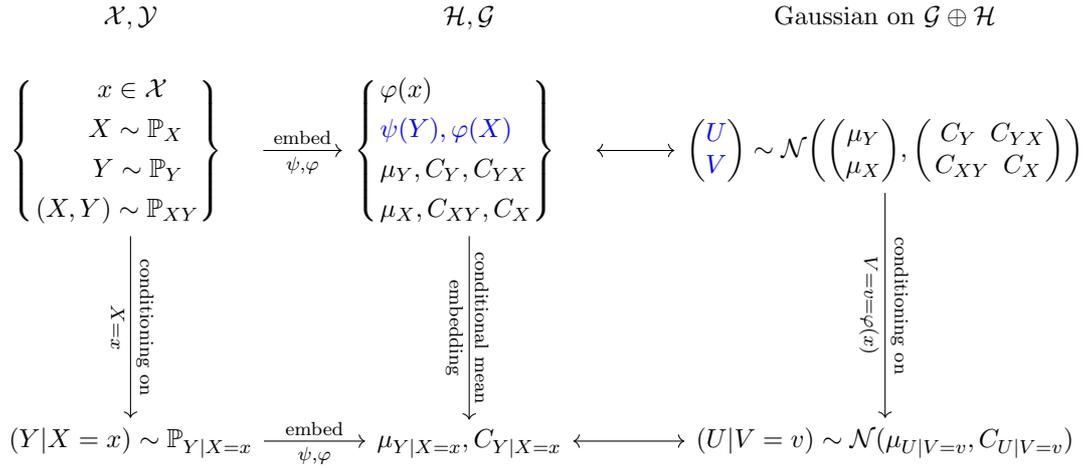

Note that the distributions of $\varphi(X)$ and $V$ might have different (and even disjoint!) supports, and so one must be particularly careful with ``almost every'' statements in this context.


\begin{theorem}
	\label{thm:ConnectionToGaussianRV}			
	Let \Cref{assumption:CME} and \Cref{assump:weakerCMElimit} hold, $(U,V)$ be the random variable defined by \eqref{equ:ChoiceUV} and let $\bP_{\varphi(X)}$ and $\bP_{V}$ denote the probability distributions of $\varphi(X)$ and $V$, respectively.
	Then, for $\bP_{V}$-a.e.\ $v\in\cH$,
	\[	
		C_{U|V=v}
		=
		\bE[C_{Y|X}]
		=
		\int_{\cX} C_{Y|X=x}\, \rd \bP_X(x).
	\]
	Further, there exist $N_1,N_2\subseteq \Omega$ with $\bP_{\varphi(X)}(N_1) = 0$ and $\bP_{V}(N_2) = 0$, such that, for every $v = \varphi(x)\notin N_1\cup N_2$, $\mu_{U|V=v} = \mu_{Y|X = x}$.	
\end{theorem}


\begin{proof}
	By \Cref{lemma:C_YgXwelldefined}, $\bE[C_{Y|X}]$ is well defined.
	The identity $\mu_{U|V=v}=\mu_{Y|X = x}$ for the means follows directly from \Cref{thm:CMEproperlyLimit,thm:MainTheoremOwhadiScovel,thm:UniqueProjectionPseudoInverse}.
	For the covariance identity, using the notation of \Cref{thm:CMEproperlyLimit}, note that $\norm{ [h_{\psi(y)}^{(n)}] - [f_{\psi(y)}] }_{ L_{\cC}^{2}}\xrightarrow[n\to\infty]{}	0$ by \eqref{equ:ConvergenceL2_C}.
	Therefore, for $y,y'\in\cY$, $g = \psi(y)$, and $g'=\psi(y')$,
	\begin{align*}
		\allowdisplaybreaks
		\Cov\left[f_{g}(X),f_{g'}(X)\right]
		&=
		\lim_{n\to\infty} \Cov\left[h_{g}^{(n)}(X), h_{g'}^{(n)}(X)\right] \\
		&=
		\lim_{n\to\infty} \innerprod{ C_{X} h_{g}^{(n)} }{ h_{g'}^{(n)} }_{\cH} \\
		&=
		\lim_{n\to\infty} \innerprod{ C_{XY}^{(n)} g }{ C_{X}^{(n)\dagger} C_{XY}^{(n)} g' }_{\cH} \\
		&=
		\lim_{n\to\infty} \innerprod{ g }{ C_{YX}^{(n)}C_{X}^{(n)\dagger}C_{XY}^{(n)} g' }_{\cG} \\
		&=
		\lim_{n\to\infty} \innerprod{ g }{ C_{UV}C_V^{(n)\dagger}C_{VU}^{(n)} g' }_{\cG} .
	\end{align*}
	By the law of total covariance and \eqref{equ:GaussianConditioningIncompatibleCase}, \eqref{UniqueObliqueEqualsPseudoInverse} this implies that, for $g = \psi(y)$ and $g'=\psi(y')$,
	\begin{align*}
		\allowdisplaybreaks
		\innerprod{ g }{ \bE[C_{Y|X}] g' }_{\cG}
		&=
		\bE\big[\Cov[g(Y),g'(Y)|X]\big] \\
		&=
		\Cov[g(Y),g'(Y)] - \Cov\big[f_g(X), f_{g'}(X)\big] \\
		&=
		\innerprod{ g }{ C_{U} g' }_{\cG} - \lim_{n\to\infty} \innerprod{ g }{ C_{UV} C_V^{(n)\dagger} C_{VU}^{(n)} g' }_{\cG} \\
		&=
		\innerprod{ g }{ C_{U|V=v}\, g' }_{\cG}
	\end{align*}
	for $\bP_{V}$-a.e.\ $v\in\cH$.
	Since $\spn\{\psi(y)\mid y\in\cY\}$ is dense in $\cG$, this finishes the proof.
\end{proof}

\begin{remark}
	\Cref{thm:ConnectionToGaussianRV} implies in particular that the posterior mean $\mu_{U|V=v}$ of the $U$-component of a jointly Gaussian random variable $(U,V)$ in an RKHS $\cG \oplus \cH$ is not just some element in $\cG$, but in fact the KME of some probability distribution on $\cY$, as long as we condition on an event of the form $V=v=\varphi(x)$ outside the null events $N_{1}$ and $N_{2}$.
	Note, though, that these null sets could be geometrically quite large.
\end{remark}

As mentioned above, there is another analogy between CMEs and Gaussian conditioning, namely the assumption under which the formula for the conditional mean is particularly nice, i.e.\ does not require finite-rank approximations of the \mbox{(cross-)}\-covariance operators:

\begin{theorem}
	\label{thm:EquivalenceAssumptionCAndCompatibility}
	Under \Cref{assumption:CME} and with the random variable $(U,V)$ defined by \eqref{equ:ChoiceUV}, \Cref{assump:weakCME} is equivalent to the compatibility of $(C,\cH)$.
\end{theorem}


\begin{proof}
	By \Cref{thm:RelationC_XandC_XY,thm:UniqueObliqueProjectionEqualsPseudoInverse}, both conditions are equivalent to $\ran C_{XY} \subseteq \ran C_{X}$.
\end{proof}

%% file: section-08-closing.tex

\section{Closing Remarks}
\label{section:Closing}

This article has demonstrated rigorous foundations for the method of conditional mean embedding in reproducing kernel Hilbert spaces.
Mild and verifiable sufficient conditions have been provided for the centred and uncentred variants of the CME formula to yield an element $\mu_{Y|X=x}$ that is indeed the kernel mean embedding of the conditional distribution $\bP_{Y|X=x}$ on $\cY$.
The CME formula required a correction in the centred case but, modulo this correction, it is more generally applicable than its uncentred counterpart and provides stronger statements:
\Cref{thm:CMEproperlyLimit} proves convergence in $L^2(\bP;\cG)$ as well as $\bP_{X}$-almost everywhere convergence, while its analogue \Cref{thm:CMEproperlyUncenteredLimit} yields only convergence in $L^2(\bP;\cG)$.
The reason is that $(\uu{\mu}^{(n)}(X,\quark))_{n\in\bN}$ defined by \eqref{equ:DefinitionU_Mu_n}, in contrast to $(\mu^{(n)}(X,\quark))_{n\in\bN}$, may fail to be a martingale (cf.\ \Cref{lemma:MartingaleConsequences}) and we cannot apply \citet[Theorem V.2.8]{diestel1977}.
Therefore, we advocate for the centred version of the CME formula as the preferred formulation in practice.
We have also demonstrated the precise relationship between CMEs and well-established formulae for the conditioning of Gaussian random variables in Hilbert spaces.

Some natural directions for further research suggest themselves:


First, in practice, the KMEs and kernel \mbox{(cross-)}\-covariance operators will often be estimated using sampled data, and so empirical versions of the CME, along with convergence guarantees, are of great practical importance.
Various empirical CMEs have already been considered and applied in the literature \citep{fukumizu2015nonparametric,fukumizu2013kernel,gruenewaelder2012conditional,park2020measure}, but their approximation accuracy is not at all trivial to analyse, conditions for validity along the lines of our Assumptions \ref{assump:strongCME}--\ref{assump:weakCMEuncentred} are not yet known, and a detailed treatment would be too long to consider in this work, which has deliberately focused on the population CME.
\Cref{section:empirical} gives an overview of the technical obstacles that must be overcome in the empirical setting, existing results in the area, and work yet to do.

Second, when using CMEs for inference, a remaining step might be to undo the kernel mean embedding, i.e.\ to recover the conditional distribution $\bP_{Y|X=x}$ on $\cY$ from its embedding $\mu_{Y|X=x} \in \cG$, or its density with respect to a reference measure on $\cY$.
This is a particular instance of a non-parametric inverse problem and a principled solution, based upon Tikhonov regularisation, has been proposed in the context of the \emph{kernel conditional density operator} (KCDO) by \citet{schuster2019kcdo}.
The relationship between this KCDO approach and the sufficient conditions for CME that have been considered in this article remains to be precisely formulated;
given the intimate relationship between Tikhonov regularisation and the Moore--Penrose pseudo-inverse, this should be a fruitful avenue of research.

%% file: chunk-acknowledgements.tex
The research presented here was supported in part by the German Research Foundation (Deut\-sche For\-schungs\-ge\-mein\-schaft) through project TrU-2 ``Demand modelling and control for e-commerce using RKHS transfer operator approaches'' of the Excellence Cluster ``MATH+ The Berlin Mathematics Research Centre'' (EXC-2046/1, project ID: 390685689).
The authors also wish to thank S.~Klus, H.~C.~Lie, M.~Mollenhauer, and B.~Sprungk for helpful and collegial discussions.

%% file: appendix-a-technical.tex

\section{Technical Results}
\label{section:TechnicalResults}

This section contains several technical results used in the proofs of the theorems given in the article.
The following well-known result due to \citet[Theorem~1]{douglas1966majorization} (see also \citet[Theorem~2.1]{fillmore1971operator}) is used several times:

\begin{theorem}
	\label{theorem:DouglasExistenceQ}
	Let $\cH$, $\cH_1$ and $\cH_2$ be Hilbert spaces and let $A\colon \cH_1\to \cH$ and $B\colon \cH_2\to \cH$ be bounded linear operators with $\ran A\subseteq \ran B$.
	Then $Q \defeq B^\dagger A\colon \cH_1 \to \cH_2$ is a well-defined bounded linear operator, where $B^\dagger$ denotes the Moore--Penrose pseudo-inverse of $B$.
	It is the unique operator that satisfies the conditions
	\begin{equation}
		A = BQ,
		\qquad
		\ker Q = \ker A,
		\qquad
		\ran Q \subseteq \overline{\ran B^{\ast}}.
	\end{equation}
\end{theorem}

\begin{remark}
	In the original work of \citet{douglas1966majorization} only the existence of a bounded operator $Q$ such that $A = BQ$ was shown.
	However, the construction of $Q$ in the proof is identical to that of $B^\dagger$ (multiplied by $A$).	
	This connection has been observed before by \citet[Corollary~2.2 and Remark~2.3]{arias2008gi_douglas}, where it was proven in the case of closed range operators, leaving the proof of the general case to the reader.
\end{remark}

The following result partially generalises \cite[Proposition 4.1]{devito2006discretization}:

\begin{lemma}
	\label{lemma:Carleman_trick}%
	Let $\cH$ be a separable Hilbert space, let $\cG$ be an RKHS over $\cY$ with canonical feature map $\psi$, and suppose that $\cG$ is a subset of $\cL^{2}(\nu)$, where $\nu$ is a $\sigma$-finite measure on $\cY$.
	Then any bounded linear operator $A \colon \cH \to \cG$ is Hilbert--Schmidt as an operator $A \colon \cH \to L^{2}(\nu)$.
\end{lemma}

\begin{proof}
	Let $h \in \cH$ and $y \in \cY$.
	Then $(A h) (y) = \innerprod{ \psi(y) }{ A h }_{\cG} = \innerprod{ A^{\ast} \psi (y) }{ h }_{\cH}$.
	Thus $A$ is a Carleman operator and the claim follows from \citet[Theorem~6.15]{weidmann1980}.
\end{proof}

The following results are used in the proofs of \Cref{section:AssumptionsForCMEs,section:TheoryCentred,section:ConnectionGaussianCME}.
Note that \Cref{lemma:EquivalenceCharacteristicDense} is essentially one direction of Proposition~5 in \citet{fukumizu2009kernel}, but does not require $k$ to be bounded, which makes a separate proof necessary.

\begin{lemma}
	\label{lemma:EquivalenceCharacteristicDense}
	Under \Cref{assumption:CME}, if $k$ is a characteristic kernel, then $\cH_{\cC}$ is dense in $ L_{\cC}^{2}(\bP_{X})$.
\end{lemma}

\begin{proof}
	Suppose that $\cH_{\cC}$ is not dense in $ L_{\cC}^{2}(\bP_{X})$.
	Then there exists $f\in L^2(\bP_{X})$ that is not $\bP_{X}$-a.e.\ constant such that $[f]\perp_{ L_{\cC}^{2}(\bP_{X})} \cH_{\cC}$.
	Choose $\tilde{f} \coloneqq f-\bE[f(X)]$ and set
	\[
	Q_1(E) \defeq \int_{E} \absval{\tilde{f}}\, \rd \bP_{X},
	\qquad
	Q_2(E) \defeq \int_{E} \bigl( \absval{\tilde{f}} - \tilde{f} \bigr) \, \rd \bP_{X}
	\]
	for every Borel-measurable subset $E\subseteq\cX$.
	Since $\norm{ \tilde{f} }_{L^1(\bP_{X})}\neq 0$, we may assume without loss of generality that $\norm{ \tilde{f} }_{L^1(\bP_{X})} = 1$, making $Q_1$ and $Q_2$ two \emph{distinct} probability distributions.
	Since, for every $h\in\cH$,
	\[
		\innerprod{ \tilde{f} }{ h }_{ L^2(\bP_{X})}
		=
		\innerprod{ f-\bE[f(X)] }{ h }_{ L^2(\bP_{X})}
		\stackrel{[f]\perp\cH_{\cC}}{=}
		\innerprod{ f-\bE[f(X)] }{ \bE[h(X)] }_{ L^2(\bP_{X})}
		=
		0,
	\]
	it follows that $\tilde{f} \perp_{ L^2(\bP_{X})} \cH$.
	Let $Z_1\sim Q_1$ and $Z_2\sim Q_2$ and $x\in\cX$.
	Since $\varphi(x)\in\cH$,
	\[
		\big( \bE[\varphi(Z_1)] - \bE[\varphi(Z_2)] \big)(x)
		=
		\innerprod{ \tilde{f} }{ \varphi(x) }_{ L^2(\bP_{X})}
		=
		0,
	\]
	which contradicts the assumption that $k$ is characteristic.
	Note that, by \Cref{assumption:CME}, $\bE[\varphi(Z_1)]$ and $\bE[\varphi(Z_2)]$ are well defined.
	In fact, by the Cauchy--Schwarz inequality,
	\[
		\bE\big[\norm{\varphi(Z_1)}_{\cH}\big]
		=
		\int_{\cX} \norm{ \varphi(x) }_{\cH}\, \absval{\tilde{f}(x)} \, \rd \bP_{X} (x)
		\leq
		\bE\big[ \norm{ \varphi(X) }_{\cH}^2 \big]^{1/2}
		\,
		\bE\big[ \tilde{f}(X)^2 \big]^{1/2}
		< \infty
	\]
	and similarly for $Z_2$.
\end{proof}
\begin{lemma}
	\label{lemma:kernelC_Xconstants}
	Under \Cref{assumption:CME}, $\ker C_{X} = \{ h\in\cH \mid h \text{ is } \bP_{X} \text{-a.e.\ constant in } \cX \}$ and $\ker \uu{C}_{X} = \{ h\in\cH \mid h=0\ \bP_{X} \text{-a.e.\ in } \cX \}$.	
\end{lemma}

\begin{proof}
	This is a direct consequence of the facts that $\innerprod{ h }{ C_{X} h }_{\cH} = \bV[h(X)]$ and that $\innerprod{ h }{ \uu{C}_{X} h }_{\cH} = \norm{ h }_{ L^2(\bP_{X})}$.
\end{proof}

\begin{lemma}
	\label{lemma:RelationCovAndC_XY}
	Under \Cref{assumption:CME}, for all $h\in\cH$ and $g\in\cG$,
	\[
	\Cov [h(X), f_g(X)]= \innerprod{ h }{ C_{XY} g }_{\cH},
	\qquad
	\uu{\Cov} [h(X), f_g(X)] = \innerprod{ h }{ \uu{C}_{XY} g }_{\cH} .
	\]
\end{lemma}

\begin{proof}
	Let $h\in\cH$ and $g\in\cG$ be arbitrary.
	Then
	\begin{align*}
		\Cov [h(X), f_g(X)]
		&=
		\bE\big[h(X)\bE[g(Y)|X]\big]
		-
		\bE[h(X)]\bE\big[\bE[g(Y)|X]\big]
		\\
		&=
		\bE[h(X)g(Y)]
		-
		\bE[h(X)]\bE[g(Y)]
		\\
		&=
		\Cov [h(X), g(Y)]
		\\
		&=
		\innerprod{ h }{ C_{XY} g }_{\cH},
	\end{align*}
	as required. The second statement is proved analogously using uncentred covariance operators and without subtracting the (products of) expected values.
\end{proof}

\begin{lemma}
	\label{lemma:TeproducingPropertyAdjoint}
	Under \Cref{assumption:CME}, let $A\colon\cG\to\cH$ be a bounded linear operator.
	Then, for all $x\in\cX$ and $y\in\cY$,
	\[
		(A\psi(y))(x)
		=
		(A^{\ast}\varphi(x))(y),
		\qquad
		\bE[(A\psi(y))(X)]
		=
		(A^{\ast}\mu_{X})(y) .
	\]
\end{lemma}

\begin{proof}
	By the reproducing properties of $\psi$, $\varphi$, and $\mu_{X}$,
	\begin{align*}
		(A\psi(y))(x)
		=
		\innerprod{ A\psi(y) }{ \varphi(x) }_{\cH}
		& =
		\innerprod{ \psi(y) }{ A^{\ast}\varphi(x) }_{\cG}
		=
		(A^{\ast}\varphi(x))(y), \\
		\bE[(A\psi(y))(X)]
		=
		\innerprod{ A\psi(y) }{ \mu_{X} }_{\cH}
		& =
		\innerprod{ \psi(y) }{ A^{\ast}\mu_{X} }_{\cG}
		=
		(A^{\ast}\mu_{X})(y) ,
	\end{align*}
	as claimed.
\end{proof}

\begin{lemma}
	\label{lemma:ConvergenceOfProjections}
	Let $V$ be a Hilbert space, let $U_1 \subseteq U_2 \subseteq \cdots$ be an increasing sequence of closed subspaces $U_n \subseteq V$, $n\in\bN$, and let $U \defeq \bigcup_{n\in\bN} U_n$.
	Further, let $P_{U_n}\colon V\to U_n$ denote the orthogonal projection onto $U_n$.
	Then, for all $v\in \overline U$,
	\[
		P_{U_n} v \xrightarrow[n \to \infty]{} v .
	\]
\end{lemma}

\begin{proof}
	Let $v\in \overline U$ and $\varepsilon>0$.
	Then there exists $u\in U$ such that $\norm{ u-v }<\varepsilon$.
	Since the sequence $(U_n)_{n \in \bN}$ is increasing and $U$ is its union, there exists an $n_0\in\bN$ such that $u\in U_n$ and thereby $P_{U_n}u = u$ for all $n\ge n_0$.
	We therefore obtain, for $n\ge n_0$,
	\[
		\norm{ P_{U_n}v-v }
		\leq
		\norm{ P_{U_n}v - P_{U_n}u } + \norm{ P_{U_n}u - u } + \norm{ u-v }
		\leq
		\norm{ P_{U_n} } \norm{v-u} + \norm{u-v}
		<
		2\varepsilon,
	\]
	by the triangle inequality and non-expansivity of orthogonal projection.
\end{proof}

\begin{lemma}
	\label{lemma:TechnicalDetailsMainProof}
	Under \Cref{assumption:CME}, with $\cH^{(n)}$, $C^{(n)}$, $h_g^{(n)}$ as in \Cref{thm:CMEproperlyLimit} and $\uu{C}^{(n)}$, $\uu{h}_g^{(n)}$ as in \Cref{thm:CMEproperlyUncenteredLimit},
	\begin{enumerate}[label=(\alph*)]
		\item \label{item:TechnicalDetailsMainProofA}
		$\displaystyle \Cov[h(X),f_g(X)] = \lim_{n\to\infty} \Cov[h(X),h_g^{(n)}(X)]$ for all $h\in\cH$;
		\item \label{item:TechnicalDetailsMainProofB}
		$[h_g^{(n)}]$ is the $ L_{\cC}^2$-orthogonal projection of $[f_g]$ onto $\cH_{\cC}^{(n)}$ for all $g\in\cG$;
		\item \label{item:TechnicalDetailsMainProofC}
		$\displaystyle \uu{\Cov}[h(X),f_g(X)] = \lim_{n\to\infty} \uu{\Cov}[h(X),\uu{h}_g^{(n)}(X)]$ for all $h\in\cH$;
		\item \label{item:TechnicalDetailsMainProofD}
		$\uu{h}_g^{(n)}$ is the $L^2$-orthogonal projection of $f_g$ onto $\cH^{(n)}$ for all $g\in\cG$.
	\end{enumerate}
\end{lemma}

\begin{proof}
	We only give the proofs of \ref{item:TechnicalDetailsMainProofA} and \ref{item:TechnicalDetailsMainProofB}; \ref{item:TechnicalDetailsMainProofC} and \ref{item:TechnicalDetailsMainProofD} can be proven similarly.
	It is clear that that $C^{(n)} \to C$ (in the strong and thereby in the weak sense) as $n \to \infty$ and that $C_{X}$ and $C_{X}^{(n)}$ agree on $\cH^{(n)} \ni h_g^{(n)}$.
	Using \Cref{lemma:RelationCovAndC_XY} we obtain, for all $h\in\cH$,
	\begin{align*}
		\allowdisplaybreaks
		\Cov[h(X),f_g(X)]
		&=
		\innerprod{ h }{ C_{XY} g }_{\cH} \\
		&=
		\lim_{n\to\infty} \innerprod{ h }{ C_{XY}^{(n)} g }_{\cH} \\
		&=
		\lim_{n\to\infty} \innerprod{ h }{ C_{X}^{(n)} h_g^{(n)} }_{\cH}	\\		
		&=
		\lim_{n\to\infty} \innerprod{ h }{ C_{X} h_g^{(n)} }_{\cH} \\
		&=
		\lim_{n\to\infty} \Cov[h(X),h_g^{(n)}(X)],
	\end{align*}
	which yields \ref{item:TechnicalDetailsMainProofA}.
	Also, for arbitrary $h^{(n)}\in \cH^{(n)}$, \Cref{lemma:RelationCovAndC_XY} yields
	\begin{align*}
		\allowdisplaybreaks
		\langle [h^{(n)}],[f_g]\rangle_{ L_{\cC}^2}
		&=
		\Cov\big[h^{(n)}(X),f_g(X)\big] \\
		&=
		\innerprod{ h^{(n)} }{ C_{XY} g }_{\cH} \\
		&=
		\innerprod{ C_{YX} h^{(n)} }{ g }_{\cG} \\
		&=
		\innerprod{ C_{YX}^{(n)} h^{(n)} }{ g }_{\cG} \\
		&=
		\innerprod{ h^{(n)} }{ C_{XY}^{(n)} g }_{\cH} \\
		&=
		\innerprod{ h^{(n)} }{ C_{X}^{(n)} h_g^{(n)} }_{\cH} \\
		&=
		\innerprod{ h^{(n)} }{ C_{X} h_g^{(n)} }_{\cH} \\
		&=
		\Cov\big[h^{(n)}(X), h_g^{(n)}(X)\big] \\
		&=
		\innerprod{ [h^{(n)}] }{ [h_g^{(n)}] }_{ L_{\cC}^2},
	\end{align*}
	which yields \ref{item:TechnicalDetailsMainProofB}.
\end{proof}

\begin{lemma}
	\label{lemma:LiftingPropertiesFromL2RtoL2G}
	Let \Cref{assumption:CME} hold and $\cH^{(1)} \subseteq \cH^{(2)} \subseteq \cdots$ be an increasing sequence of closed subspaces $\cH^{(n)}$ of $ L^{2}(\bP_{X})$, $n\in\bN$.
	Further, let $\fm,\fm^{(n)} \in L^{2}(\bP_{X} ; \cG) \simeq L^{2}(\bP_{X})\otimes \cG$ and denote $\overline{f} \defeq f - \bE[f(X)]$ for $f \in L^{2}(\bP_{X})$ and $\overline{\ff} \defeq \ff - \bE[\ff(X)]$ for $\ff \in L^{2}(\bP_{X} ; \cG)$.
	\begin{enumerate}[label=(\alph*)]		
		\item \label{item:LiftingPropertyA}
		If
		$([\fm^{(n)}](\quark))(y)$ is the orthogonal projection in $ L_{\cC}^{2}(\bP_{X})$ of $([\fm](\quark))(y)$ onto $\cH_{\cC}^{(n)}$ for each $y\in\cY$,
		then
		$[\fm^{(n)}]$ is the orthogonal projection in $ L_{\cC}^{2}(\bP_{X} ; \cG)$ of $[\fm]$ onto $(\cH^{(n)}\otimes \cG)_{\cC}$.		
		\item \label{item:LiftingPropertyB}
		If, in addition to the assumption in \ref{item:LiftingPropertyA},
		$([\fm^{(n)}](\quark))(y) \to ([\fm](\quark))(y)$ in $ L_{\cC}^{2}(\bP_{X})$ as $n \to \infty$ for each $y\in\cY$,
		then
		$[\fm^{(n)}] \to [\fm]$ in $ L_{\cC}^{2}(\bP_{X} ; \cG)$,
		or, in other words, 
		$\overline{\fm}^{(n)}(X) \to \overline{\fm}(X)$ in $ L^{2}(\bP ; \cG)$.
		\item \label{item:LiftingPropertyC}
		If
		$(\fm^{(n)}(\quark))(y)$ is the orthogonal projection in $ L^{2}(\bP_{X} ; \bR)$ of $(\fm(\quark))(y)$ onto $\cH^{(n)}$ for each $y\in\cY$,
		then
		$\fm^{(n)}$ is the orthogonal projection in $ L^{2}(\bP_{X} ; \cG)$ of $\fm$ onto $\cH^{(n)}\otimes \cG$.
		\item \label{item:LiftingPropertyD}
		If, in addition to the assumption in \ref{item:LiftingPropertyC}, 
		$(\fm^{(n)}(\quark))(y) \to (\fm(\quark))(y)$ in $ L^{2}(\bP_{X} ; \bR)$ as $n \to \infty$ for each $y\in\cY$,
		then
		$\fm^{(n)} \to \fm$ in $ L^{2}(\bP_{X} ; \cG)$,
		or, in other words, 
		$\fm^{(n)}(X) \to \fm(X)$ in $ L^{2}(\bP ; \cG)$.
	\end{enumerate}
\end{lemma}

\begin{proof}
	We only give the proofs of \ref{item:LiftingPropertyA} and \ref{item:LiftingPropertyB};
	\ref{item:LiftingPropertyC} and \ref{item:LiftingPropertyD} can be proven similarly with fewer technicalities.	
	Let $h\in\cH^{(n)}$ and $y\in\cY$. Then
	\begin{align*}
		\innerprod{[\fm^{(n)}] - [\fm]}{[h\otimes \psi(y)]}_{ L_{\cC}^{2}(\bP_{X} ; \cG)}
		&=
		\bE\bigl[ \biginnerprod{\overline{\fm}^{(n)}(X) - \overline{\fm}(X)}{\psi(y)}_{\cG}\, \overline{h}(X) \bigr]
		\\
		&=
		\bE \left[ \Bigl( \bigl(\overline{\fm}^{(n)}(X)\bigr)(y) - \big(\overline{\fm} (X)\big)(y) \Bigr) \, \overline{h}(X) \right]
		\\
		&=
		\biginnerprod{ \big([\fm^{(n)}](\quark)\big)(y) - \big([\fm] (\quark)\big)(y) }{[h]}_{ L_{\cC}^{2}(\bP_{X};\bR)}
		\\
		&=
		0,
	\end{align*}
	which proves \ref{item:LiftingPropertyA}.
	Hence, by \Cref{lemma:ConvergenceOfProjections}, $[\fm^{(n)}]$ converges in $ L_{\cC}^{2}(\bP_{X} ; \cG)$ to some limit $[\fm']$.
	This implies pointwise convergence for each $y\in\cY$ in the following sense:
	\begin{align*}
		\bignorm{ \bigl([\fm^{(n)}](\quark)\big)(y) - \big([\fm'](\quark)\bigr)(y) }_{ L_{\cC}^{2}(\bP_{X} ; \bR)}^{2}
		&=
		\bE \bigl[ \absval{ \innerprod{\psi(y)}{\overline{\fm}^{(n)}(X) - \overline{\fm}'(X)}_{\cG} }^{2} \bigr]
		\\
		&\leq
		\norm{\psi(y)}_{\cG}^{2} \, \bE \bigl[ \norm{\overline{\fm}^{(n)}(X) - \overline{\fm}'(X)}_{\cG}^{2} \bigr]
		\\
		&=
		\norm{\psi(y)}_{\cG}^{2} \, \norm{\fm^{(n)} - \fm'}_{ L_{\cC}^{2}(\bP_{X} ; \cG)}^{2}
		\\
		&
		\xrightarrow[n\to\infty]{}0.
	\end{align*}
	Therefore, by assumption, $[\fm']$ agrees with $[\fm]$ $\bP_{X}$-a.e., proving \ref{item:LiftingPropertyB}.
\end{proof}
\begin{lemma}
	\label{lemma:MartingaleConsequences}
	Under the assumptions and notation of \Cref{thm:CMEproperlyLimit}, $(\mu^{(n)}(X,\quark))_{n\in\bN}$ is a martingale in $ L^{2}(\Omega,\Sigma,\bP ; \cG)$ with respect to the filtration $(\sigma(V^{(n)}))_{n\in\bN}$ of $\Sigma$, where $V^{(n)} \defeq P_{\cH^{(n)}}(\varphi(X))$ and $P_{\cH^{(n)}}\colon \cH \to \cH^{(n)}$ denotes the orthogonal projection in $\cH$ onto $\cH^{(n)}$.
\end{lemma}

\begin{proof}
	Consider the Karhunen--Lo\`eve expansion of $\varphi(X)$,
	\[
	\varphi(X) = \mu_X + \sum_{i\in\bN} Z_{i}\, h_{i},
	\]
	where $Z_{i}\colon (\Omega,\Sigma,\bP)\to\bR$ are uncorrelated real-valued random variables with $\bE[Z_{i}]=0$ and $\bV[Z_{i}]=\sigma_{i}$ for all $i\in\bN$, $\sigma_i \ge 0$ denoting the eigenvalue of $C_{X}$ corresponding to the eigenvector $h_i$.
	We observe that $\sigma(V^{(n)}) = \sigma(Z^{(n)})$, where $Z^{(n)}\defeq (Z_1,\dots,Z_n)$.
	Now let $A^{(n)} = ( C_{X}^{(n)\dagger} C_{XY}^{(n)})^{\ast}$, $n\in\bN$, and observe that $A^{(n)}v = A^{(n)}P_{\cH^{(n)}}v$ and that $A^{(n+1)}$ and $A^{(n)}$ agree on $\cH^{(n)}$.
	Hence, for $n\in\bN$,
	\begin{align*}
		\bE[ \mu^{(n+1)}(X,\quark) \, | \, V^{(n)}]
		&=
		\mu_Y + A^{(n+1)}\, \bE[V-\mu_X \, | \, Z^{(n)} ]
		\\
		&=
		\mu_Y + A^{(n+1)}\, \bE\Bigl[\sum_{i\in \bN} Z_{i}\, h_{i} \, \Big| \, Z^{(n)} \Bigr]
		\\
		&=
		\mu_Y + A^{(n)}\, \sum_{i=1}^{n} Z_{i}\, h_{i}
		\\
		&=
		\mu_Y + A^{(n)}\, (\varphi(X)-\mu_X)
		\\
		&=
		\mu^{(n)}(X,\quark),
	\end{align*}
	proving the martingale property.
\end{proof}

\begin{lemma}
	\label{lemma:C_YgXwelldefined}
	Let Assumptions \ref{assumption:CME} and \ref{assump:weakerCMElimit} hold.
	Then $\bE[C_{Y|X}] = \int_{\cX} C_{Y|X=x}\, \rd \bP_{X}(x)$ is well defined as a strong (Bochner) integral, i.e.\ $\int_\cX \norm{ C_{Y|X=x} } \, \rd \bP_{X}(x) <\infty$.
\end{lemma}

\begin{proof}
	The Cauchy--Schwarz inequality and \eqref{equ:HnormStrongerThanL2norm} imply that, for $x\in\cX_{Y}$,
	\begin{align*}
		\norm{ C_{Y|X=x} } 
		& =
		\sup_{\norm{ g }_{\cG} \leq 1, \norm{ \tilde{g} }_{\cG} \leq 1} \innerprod{ g }{ C_{Y|X=x}\tilde{g} }_{\cG} \\
		& =
		\sup_{\norm{ g }_{\cG} \leq 1, \norm{ \tilde{g} }_{\cG} \leq 1} \innerprod{ g }{ \tilde{g} }_{\cL_{\cC}^{2}(\bP_{Y|X=x})} \\
		& \leq
		\sup_{\norm{ g }_{\cG} \leq 1, \norm{ \tilde{g} }_{\cG} \leq 1} \norm{ g }_{\cL^2(\bP_{Y|X=x})} \norm{ \tilde{g} }_{\cL^2(\bP_{Y|X=x})} \\
		& \leq
		\bE\bigl[\norm{\psi(Y)}_{\cG}^2 \big| X=x\bigr],
	\end{align*}
	which, by the law of total expectation and \eqref{equ:FiniteSecondMomentOfUandV}, yields that
	\[
		\bE[ \norm{ C_{Y|X} } ]
		\leq
		\bE\big[\bE[\norm{\psi(Y)}_{\cG}^2 \mid X]\big]
		=
		\bE[\norm{\psi(Y)}_{\cG}^2]
		<
		\infty,
	\]
	as claimed.
\end{proof}

%% file: appendix-c-empirical.tex

\section{Empirical Estimates for CMEs}
\label{section:empirical}

In practice, the kernel mean embeddings and kernel \mbox{(cross-)}\-covariance operators will often be estimated empirically from observed data, and so empirical versions of the CME, along with convergence guarantees, are of great importance.
As mentioned in \Cref{remark:EmpiricalIsHard}, this topic is beyond the scope of this paper.
However, we wish to point out why this is a complex problem and briefly address the main difficulties.

In the simplest setting, given $J \in \bN$ independent samples $(X_{1}, Y_{1}), \dots, (X_{J}, Y_{J}) \sim \bP_{XY}$, we have the empirical estimators
\begin{align*}
	\mu_{X} & \approx \widehat{\mu}_{X} \defeq \frac{1}{J} \sum_{j = 1}^{J} \varphi(X_{j}) , &
	C_{XY} & \approx \widehat{C}_{XY} \defeq \frac{1}{J} \sum_{j = 1}^{J} ( \varphi(X_{j}) - \widehat{\mu}_{X} ) \otimes ( \psi(Y_{j}) - \widehat{\mu}_{Y} ) ,
\end{align*}
and so on.
(To simplify the notation, we suppress the obvious $J$-dependence of these estimators.)
Laws of large numbers for these empirical estimators have already been established --- see e.g.\ \citet[Theorem~2]{smola2007embedding} and \citet[Lemma~5.8]{mollenhauer2018ma} --- but the impact of this approximation error upon conditioning is, to the best of our knowledge, not yet fully quantified.
One natural approach to approximate the CME $\mu_{Y|X = x}$ is the regularisation of $\widehat{C}_{X}$ or $\widehat{\uu{C}}_{X}$,
\[
	\mu_{Y|X = x}
	\approx
	\Big(\bigl( \widehat{\uu{C}}_{X} + \varepsilon \Id_{\cH}\bigr)^{\dagger}\, \widehat{\uu{C}}_{XY}\Big)^{\ast} \varphi(x)
	=
	\widehat{\uu{C}}_{YX} \bigl( \widehat{\uu{C}}_{X} + \varepsilon \Id_{\cH}\bigr)^{-1} \varphi(x),
\]
where $\varepsilon > 0$ is a regularisation parameter which may depend on $J$.
Note that such a regularisation can be viewed as an approximation both to the new CME formula derived in \Cref{thm:CMEproperlyUncentered}, $\mu_{Y|X = x} = (\uu{C}_{X}^{\dagger} \uu{C}_{XY})^{\ast} \varphi(x)$, as well as to the original (uncentred) one, $\mu_{Y|X = x} = \uu{C}_{YX} \uu{C}_{X}^{-1} \varphi(x)$.
Therefore, this approach is rather well studied and convergence rates for this strategy have been established under certain conditions \citep{fukumizu2015nonparametric,gruenewaelder2012conditional,park2020measure}.

However, the new formulae \eqref{equ:CMEmuNEW}, \eqref{equ:CMElimit}, and \eqref{equ:CMEmuNEWAlternativeUncentred} relying on the Moore--Penrose pseudo-inverse suggest another type of approximation, where we will focus on the centred case from now on.
The na{\"\i}ve estimate would be
\begin{equation}
	\label{equ:naive_empirical_CME}
	\mu_{Y|X = x}
	\approx
	\widehat{\mu}_{Y} + \bigl( \widehat{C}_{X}^{\dagger} \widehat{C}_{XY} \bigr)^{\ast} \, \bigl( \varphi(x) - \widehat{\mu}_{X} \bigr).
\end{equation}
Note that $\ran \widehat{C}_{XY} \subseteq \ran \widehat{C}_{X}$ and so \eqref{equ:naive_empirical_CME} is well defined.
However, the convergence of $\widehat{C}_{X}$ to $C_{X}$ (e.g.\ in the Hilbert--Schmidt norm, as $J \to \infty$) translates badly to the convergence of $\widehat{C}_{X}^{\dagger}$ to the pseudo-inverse $C_{X}^{\dagger}$.
One problem is that small eigenvalues of $C_{X}$ might be approximated by eigenvalues of $\widehat{C}_{X}$ that are orders of magnitude smaller, causing $\widehat{C}_{X}^{\dagger}$ to ``blow up''.
So, in addition to the convergence of $\widehat{C}_{X}$ in the classical norms (such as the Hilbert--Schmidt norm or operator norm), we need to control the the smallest eigenvalue of $\widehat{C}_{X}$.

A natural workaround, inspired by the finite-rank approximation in \Cref{thm:CMEproperlyLimit}, is to truncate\footnote{Naturally, truncation can be viewed as another form of regularisation. For further regularised estimates of large covariance and precision matrices by tapering, banding, sparsifying or similar see e.g.\ \citet{bickel2008covariance,bickel2008regularized,cai2010optimal,yuan2010high} and references therein.}
the \mbox{(cross-)}\-covariance operators to a subspace $\cH^{(n)} = \spn\{ h_1,\dots,h_n \}$ of $\cH$ with $\dim \cH^{(n)} = n = n(J) \ll J$.
One might thus hope to approximate the dominant $n$ eigenvalues of $C_{X}$ well while artificially setting the others to zero and preventing the blow-up of $\widehat{C}_{X}^{\dagger}$.
There are several results from random matrix theory that control the behaviour of the $n$\textsuperscript{th} eigenvalue of (truncated) empirical covariance matrices for growing $J$ and $n = n(J)$ \citep{bai2010spectral,bai1999methodologies,bai1993limit,heiny2018almost}.
Most of these results are formulated for the case where the true mean is known to be zero and the true covariance matrix is the identity matrix and are typically of the following form, where $\lambda_{\textup{max}}(M)$ and $\lambda_{\textup{min}}(M)$ denote the largest and smallest eigenvalues of a matrix $M$, respectively:

\begin{theorem}[{\citet[Theorem 2]{bai1993limit}}]%
	\label{theorem:Bai1993smallestEigenvalue}%
	Let $(\xi_{ij})_{i,j\in\bN}$ be a double array of independent and identically distributed random variables with zero mean and unit variance.
	For $J\in\bN$, let $n = n(J)$ be such that $n(J) \to \infty$ and $n(J)/J \to \gamma\in (0,1)$ for $J\to \infty$ and let
	\begin{equation}
		\label{equ:DefinitionA_JandS_J}
		A_J = (\xi_{ij})_{i = 1,\dots,n(J),\, j = 1,\dots,J},
		\qquad
		S_J = \tfrac{1}{J} A_JA_J^{\top}.
	\end{equation}
	Then, if $\bE[\xi_{11}^4] <\infty$,
	\[
		\lambda_{\textup{max}}(S_J) \xrightarrow[J\to\infty]{\text{a.e.}} (1+\sqrt{\gamma})^2,
		\qquad
		\lambda_{\textup{min}}(S_J) \xrightarrow[J\to\infty]{\text{a.e.}} (1-\sqrt{\gamma})^2.
	\]
\end{theorem}

In our case the covariance operator $C_{X}$ is not the identity;
however, our case follows partially from \Cref{theorem:Bai1993smallestEigenvalue} using the eigendecomposition of $C_X$,
\[
	C_X = \sum_{i\in\bN}\sigma_i \, h_i\otimes h_i,
	\qquad
	\sigma_1 \geq \sigma_2 \geq \cdots \geq 0,
\]
and the Karhunen--Lo\`eve expansion of $V = \varphi(X) - \mu_X$,
\[
	V = \sum_{i \in \bN} \sqrt{\sigma_{i}}\, \xi_{i}\, h_{i},
\]
where $\xi_{n} \colon \Omega\to\bR$ are uncorrelated random variables with $\bE[\xi_{n}]=0$ and $\bV[\xi_{i}]=1$ for each $i\in\bN$ and $(h_i)_{i\in\bN}$ is an orthonormal eigenbasis of $\cH$.
To this end, let $V_j \defeq \varphi(X_j) - \mu_X$ be i.i.d.\ copies of $V$ and let $V_j^{(n)}$ be their respective orthogonal projections onto $\cH^{(n)} \defeq \spn \{ h_1,\dots,h_n \}$, i.e.
\[
	V_j = \sum_{i \in \bN} \sqrt{\sigma_{i}}\, \xi_{ij}\, h_{i} \text{ with } \xi_{ij}\stackrel{\text{i.i.d.}}{\sim}\xi_{i} \text{ for all } i,j,
	\qquad
	V_j^{(n)} = \sum_{i =1}^{n} \sqrt{\sigma_{i}}\, \xi_{ij}\, h_{i}.
\]
To simplify notation, let us work with $(n\times n)$-matrices instead of operators, expressed in the basis $(h_1,\dots,h_n)$ of $\cH^{(n)}$. Then
\begin{equation}
	\label{equ:DecompositionOfHatC_Xn}
	\widehat{C}_{X}^{(n)}
	=
	\frac{1}{J} \sum_{j = 1}^{J} V_j \otimes V_j
	=
	\big( C_{X}^{(n)} \big)^{1/2} \, S_J \, \big( C_{X}^{(n)} \big)^{1/2},
\end{equation}
where $S_J$ is defined by \eqref{equ:DefinitionA_JandS_J}.
So, we are \emph{nearly} in the setup of \Cref{theorem:Bai1993smallestEigenvalue} and ready to conclude
\[
	\lim_{J\to\infty} \lambda_{\textup{min}} (\widehat{C}_{X}^{(n)})
	\ge
	\lim_{J\to\infty} \sigma_{n}^{1/2} \lambda_{\textup{min}}(S_J) \sigma_{n}^{1/2}
	\ge
	\sigma_{n} (1-\sqrt{\gamma})^2,
\]
with $n = n(J)$ and $\gamma$ as in \Cref{theorem:Bai1993smallestEigenvalue}.
(Here we make use of the fact that $\norm{ A x } \geq \lambda_{\textup{min}}(A) \norm{ x }$ for all $x$ when $A$ is positive semi-definite.)
However, some obstacles remain:
\begin{itemize}
	\item Since the random variables $V_j$ are independent copies of $V$, the distributions of $\xi_{ij}$ and $\xi_{i j'}$ agree for all $j,j'$, but we cannot expect the distributions of $\xi_{ij}$ and $\xi_{i' j}$ to agree for all $i, i'$.
	Hence, $\xi_{ij}$ do not fulfil the requirement of being identically distributed.
	However, there are generalisations of \Cref{theorem:Bai1993smallestEigenvalue} to this case under technical assumptions;
	see e.g. \citet[Theorem 2.8]{bai1999methodologies}.
	\item It is unclear when the condition $\bE[\xi_{11}^4] <\infty$ is fulfilled.
	There are, however, some results in the case of infinite fourth moments;
	see \citet{heiny2018almost} and references therein.
	\item The random variables $\xi_i$ are uncorrelated but, in general, not independent.
	We are not aware of a result similar to \Cref{theorem:Bai1993smallestEigenvalue} for the uncorrelated case.\footnote{\citet{chen2013covariance} discuss the estimation of covariance and precision matrices for time series data where the random variables $V_j$ are not assumed to be independent.
	However, we need to drop independence in the rows of $A_J$, not in its columns.}
\end{itemize}
Even if we manage to formulate a version of \Cref{theorem:Bai1993smallestEigenvalue} which suits our needs, note that the considerations so far, if executed rigorously, only guarantee that $C_{X}^{\dagger}$ does not blow up.
This is still some way short of establishing the convergence of the corresponding CME estimator
\begin{equation}
	\label{equ:truncated_empirical_CME}
	\mu_{Y|X = x}
	\approx
	\widehat{\mu}_{Y} + \bigl( \widehat{C}_{X}^{(n)\dagger} \widehat{C}_{XY}^{(n)} \bigr)^{\ast} \, \bigl( \varphi(x) - \widehat{\mu}_{X} \bigr),
	\qquad
	n = n(J).
\end{equation}
We conclude here by summarising some of the necessary steps:
\begin{itemize}
	\item In the above considerations we assumed $\mu_{X}$ to be known.
	In practice, however, we may only employ its empirical estimate $\widehat{\mu}_{X}$.
	Since $\mu_X\otimes \mu_X$ is a rank-one estimator, this issue might be partially resolved by \citet[Corollary 2.1]{gohberg1969introduction}, which implies that the eigenvalues of $\uu{C}_X$ and $C_X = \uu{C}_X - \mu_X\otimes \mu_X$ have a similar decay rate.
	\item In order to project onto $\cH^{(n)} = \spn \{ h_1,\dots,h_n \}$, we require the eigenvectors $h_i$ of $C_X$.
	While the $n$ dominant eigenvectors of $\widehat{C}_{X}$ can be used as estimates of $h_i$, it is unclear how the approximation error affects the theoretical results presented above.
	\item Controlling the eigenvalues alone is insufficient.
	For a reasonable approximation of $C_{X}^{(n)}$ in \eqref{equ:DecompositionOfHatC_Xn}, we would need a result which tells us that $S_J$ becomes close to the identity matrix for large $J$, which requires us to understand the behaviour of its eigenvectors as well.
	The investigation of the eigenvectors of $S_J$ turns out to be extremely challenging;
	see \citet[Chapter 10]{bai2010spectral} for a survey of existing results.
\end{itemize}

%% file: rigorousCME.bbl
\begin{thebibliography}{43}
\providecommand{\natexlab}[1]{#1}
\providecommand{\url}[1]{\texttt{#1}}
\expandafter\ifx\csname urlstyle\endcsname\relax
  \providecommand{\doi}[1]{doi: #1}\else
  \providecommand{\doi}{doi: \begingroup \urlstyle{rm}\Url}\fi

\bibitem[Arias et~al.(2008)Arias, Corach, and Gonzalez]{arias2008gi_douglas}
M.~L. Arias, G.~Corach, and M.~C. Gonzalez.
\newblock Generalized inverses and {Douglas} equations.
\newblock \emph{Proc. Amer. Math. Soc.}, 136\penalty0 (9):\penalty0 3177--3183,
  2008.
\newblock URL \url{https://doi.org/10.1090/S0002-9939-08-09298-8}.

\bibitem[Bai and Silverstein(2010)]{bai2010spectral}
Z.~Bai and J.~W. Silverstein.
\newblock \emph{Spectral {Analysis} of {Large} {Dimensional} {Random}
  {Matrices}}.
\newblock Springer Series in Statistics. Springer, New York, second edition,
  2010.
\newblock URL \url{https://doi.org/10.1007/978-1-4419-0661-8}.

\bibitem[Bai(1999)]{bai1999methodologies}
Z.~D. Bai.
\newblock Methodologies in spectral analysis of large-dimensional random
  matrices, a review.
\newblock \emph{Statist. Sinica}, 9\penalty0 (3):\penalty0 611--677, 1999.
\newblock \doi{10.1142/9789812793096_0015}.
\newblock With comments by G. J. Rodgers and Jack W. Silverstein; and a
  rejoinder by the author.

\bibitem[Bai and Yin(1993)]{bai1993limit}
Z.~D. Bai and Y.~Q. Yin.
\newblock Limit of the smallest eigenvalue of a large-dimensional sample
  covariance matrix.
\newblock \emph{Ann. Probab.}, 21\penalty0 (3):\penalty0 1275--1294, 1993.
\newblock URL \url{https://doi.org/10.1214/aop/1176989118}.

\bibitem[Baker(1973)]{baker1973joint}
C.~R. Baker.
\newblock Joint measures and cross-covariance operators.
\newblock \emph{Trans. Amer. Math. Soc.}, 186:\penalty0 273--289, 1973.
\newblock URL \url{https://doi.org/10.2307/1996566}.

\bibitem[Berlinet and Thomas-Agnan(2004)]{berlinet2004rkhs}
A.~Berlinet and C.~Thomas-Agnan.
\newblock \emph{Reproducing {Kernel} {Hilbert} {Spaces} in {Probability} and
  {Statistics}}.
\newblock Springer, Boston, 2004.
\newblock URL \url{https://doi.org/10.1007/978-1-4419-9096-9}.

\bibitem[Bickel and Levina(2008{\natexlab{a}})]{bickel2008covariance}
P.~J. Bickel and E.~Levina.
\newblock Covariance regularization by thresholding.
\newblock \emph{Ann. Statist.}, 36\penalty0 (6):\penalty0 2577--2604,
  2008{\natexlab{a}}.
\newblock URL \url{https://doi.org/10.1214/08-AOS600}.

\bibitem[Bickel and Levina(2008{\natexlab{b}})]{bickel2008regularized}
P.~J. Bickel and E.~Levina.
\newblock Regularized estimation of large covariance matrices.
\newblock \emph{Ann. Statist.}, 36\penalty0 (1):\penalty0 199--227,
  2008{\natexlab{b}}.
\newblock URL \url{https://doi.org/10.1214/009053607000000758}.

\bibitem[Cai et~al.(2010)Cai, Zhang, and Zhou]{cai2010optimal}
T.~T. Cai, C.-H. Zhang, and H.~H. Zhou.
\newblock Optimal rates of convergence for covariance matrix estimation.
\newblock \emph{Ann. Statist.}, 38\penalty0 (4):\penalty0 2118--2144, 2010.
\newblock URL \url{https://doi.org/10.1214/09-AOS752}.

\bibitem[Chen et~al.(2013)Chen, Xu, and Wu]{chen2013covariance}
X.~Chen, M.~Xu, and W.~B. Wu.
\newblock Covariance and precision matrix estimation for high-dimensional time
  series.
\newblock \emph{Ann. Statist.}, 41\penalty0 (6):\penalty0 2994--3021, 2013.
\newblock URL \url{https://doi.org/10.1214/13-AOS1182}.

\bibitem[Corach et~al.(2001)Corach, Maestripieri, and
  Stojanoff]{corach2000oblique}
G.~Corach, A.~Maestripieri, and D.~Stojanoff.
\newblock Oblique projections and {Schur} complements.
\newblock \emph{Acta Sci. Math. (Szeged)}, 67\penalty0 (1-2):\penalty0
  337--356, 2001.

\bibitem[De~Vito et~al.(2006)De~Vito, Rosasco, and
  Caponnetto]{devito2006discretization}
E.~De~Vito, L.~Rosasco, and A.~Caponnetto.
\newblock Discretization error analysis for {Tikhonov} regularization.
\newblock \emph{Analysis and Applications}, 04\penalty0 (01):\penalty0 81--99,
  2006.
\newblock URL \url{https://doi.org/10.1142/S0219530506000711}.

\bibitem[Diestel and Uhl(1977)]{diestel1977}
J.~Diestel and J.~J. Uhl.
\newblock \emph{Vector {Measures}}, volume~15 of \emph{Mathematical Surveys}.
\newblock American Mathematical Society, Providence, RI, 1977.
\newblock URL \url{https://doi.org/10.1090/surv/015}.

\bibitem[Douglas(1966)]{douglas1966majorization}
R.~G. Douglas.
\newblock On majorization, factorization, and range inclusion of operators on
  {Hilbert} space.
\newblock \emph{Proc. Amer. Math. Soc.}, 17:\penalty0 413--415, 1966.
\newblock URL \url{https://doi.org/10.2307/2035178}.

\bibitem[Engl et~al.(1996)Engl, Hanke, and Neubauer]{engl1996regularization}
H.~W. Engl, M.~Hanke, and A.~Neubauer.
\newblock \emph{Regularization of {Inverse} {Problems}}, volume 375 of
  \emph{Mathematics and its Applications}.
\newblock Kluwer Academic Publishers Group, Dordrecht, 1996.

\bibitem[Fillmore and Williams(1971)]{fillmore1971operator}
P.~A. Fillmore and J.~P. Williams.
\newblock On operator ranges.
\newblock \emph{Adv. Math.}, 7\penalty0 (3):\penalty0 254--281, 1971.
\newblock URL \url{https://doi.org/10.1016/S0001-8708(71)80006-3}.

\bibitem[Fukumizu(2015)]{fukumizu2015nonparametric}
K.~Fukumizu.
\newblock Nonparametric {Bayesian} inference with kernel mean embedding.
\newblock In \emph{Modern Methodology and Applications in Spatial-Temporal
  Modeling}, pages 1--24. Springer Japan, Tokyo, 2015.
\newblock URL \url{https://doi.org/10.1007/978-4-431-55339-7\_1}.

\bibitem[Fukumizu et~al.(2004{\natexlab{a}})Fukumizu, Bach, and
  Jordan]{fukumizu2004dimensionality}
K.~Fukumizu, F.~R. Bach, and M.~I. Jordan.
\newblock Dimensionality reduction for supervised learning with reproducing
  kernel {Hilbert} spaces.
\newblock \emph{J. Mach. Learn. Res.}, 5\penalty0 (Jan):\penalty0 73--99,
  2004{\natexlab{a}}.
\newblock URL
  \url{http://www.jmlr.org/papers/volume5/fukumizu04a/fukumizu04a.pdf}.

\bibitem[Fukumizu et~al.(2004{\natexlab{b}})Fukumizu, Bach, and
  Jordan]{fukumizu2004dimensionality_erratum}
K.~Fukumizu, F.~R. Bach, and M.~I. Jordan.
\newblock Erratum: {Dimensionality} reduction for supervised learning with
  reproducing kernel {Hilbert} spaces.
\newblock \emph{J. Mach. Learn. Res.}, 2004{\natexlab{b}}.
\newblock URL
  \url{http://www.jmlr.org/papers/volume5/fukumizu04a/fukumizu04a-erratum.pdf}.

\bibitem[Fukumizu et~al.(2008)Fukumizu, Gretton, Sun, and
  Sch\"{o}lkopf]{fukumizu2008nips}
K.~Fukumizu, A.~Gretton, X.~Sun, and B.~Sch\"{o}lkopf.
\newblock Kernel measures of conditional dependence.
\newblock In \emph{Advances in Neural Information Processing Systems 20}, pages
  489--496. Curran Associates, Inc., 2008.
\newblock URL
  \url{https://papers.nips.cc/paper/3340-kernel-measures-of-conditional-dependence.pdf}.

\bibitem[Fukumizu et~al.(2009)Fukumizu, Bach, and Jordan]{fukumizu2009kernel}
K.~Fukumizu, F.~R. Bach, and M.~I. Jordan.
\newblock Kernel dimension reduction in regression.
\newblock \emph{Ann. Statist.}, 37\penalty0 (4):\penalty0 1871--1905, 2009.
\newblock URL \url{https://doi.org/10.1214/08-AOS637}.

\bibitem[Fukumizu et~al.(2013)Fukumizu, Song, and Gretton]{fukumizu2013kernel}
K.~Fukumizu, L.~Song, and A.~Gretton.
\newblock Kernel {Bayes}' rule: {Bayesian} inference with positive definite
  kernels.
\newblock \emph{J. Mach. Learn. Res.}, 14\penalty0 (1):\penalty0 3753--3783,
  2013.
\newblock URL
  \url{http://jmlr.org/papers/volume14/fukumizu13a/fukumizu13a.pdf}.

\bibitem[Gohberg and Kre\u{\i}n(1969)]{gohberg1969introduction}
I.~C. Gohberg and M.~G. Kre\u{\i}n.
\newblock \emph{Introduction to the {Theory} of {Linear} {Nonselfadjoint}
  {Operators}}, volume~18 of \emph{Translations of Mathematical Monographs}.
\newblock American Mathematical Society, Providence, R.I., 1969.
\newblock Transl. A.\ Feinstein.

\bibitem[Gr{\"u}new{\"a}lder et~al.(2012)Gr{\"u}new{\"a}lder, Lever,
  Baldassarre, Patterson, Gretton, and Pontil]{gruenewaelder2012conditional}
S.~Gr{\"u}new{\"a}lder, G.~Lever, L.~Baldassarre, S.~Patterson, A.~Gretton, and
  M.~Pontil.
\newblock Conditional mean embeddings as regressors.
\newblock In \emph{Proceedings of the 29th International Conference on Machine
  Learning}, pages 1823--1830, 2012.
\newblock URL \url{https://icml.cc/2012/papers/898.pdf}.

\bibitem[Heiny and Mikosch(2018)]{heiny2018almost}
J.~Heiny and T.~Mikosch.
\newblock Almost sure convergence of the largest and smallest eigenvalues of
  high-dimensional sample correlation matrices.
\newblock \emph{Stochastic Process. Appl.}, 128\penalty0 (8):\penalty0
  2779--2815, 2018.
\newblock URL \url{https://doi.org/10.1016/j.spa.2017.10.002}.

\bibitem[Kallenberg(2006)]{kallenberg2006foundations}
O.~Kallenberg.
\newblock \emph{Foundations of {Modern} {Probability}}.
\newblock Springer, New York, 2006.
\newblock URL \url{https://doi.org/10.1007/978-1-4757-4015-8}.

\bibitem[Mollenhauer(2018)]{mollenhauer2018ma}
M.~Mollenhauer.
\newblock Singular value decomposition of operators on reproducing kernel
  {Hilbert} spaces.
\newblock Master's thesis, Freie Universit\"at Berlin, 2018.

\bibitem[Owhadi and Scovel(2017)]{owhadi2017separability}
H.~Owhadi and C.~Scovel.
\newblock Separability of reproducing kernel spaces.
\newblock \emph{Proc. Amer. Math. Soc.}, 145\penalty0 (5):\penalty0 2131--2138,
  2017.
\newblock URL \url{https://doi.org/10.1090/proc/13354}.

\bibitem[Owhadi and Scovel(2018)]{owhadi2015conditioning}
H.~Owhadi and C.~Scovel.
\newblock Conditioning {Gaussian} measure on {Hilbert} space.
\newblock \emph{J. Math. Stat. Anal.}, 2018.
\newblock URL \url{https://arxiv.org/abs/1506.04208}.

\bibitem[Park and Muandet(2020)]{park2020measure}
J.~Park and K.~Muandet.
\newblock A measure-theoretic approach to kernel conditional mean embeddings,
  2020.
\newblock URL \url{https://arxiv.org/abs/2002.03689}.

\bibitem[Saitoh and Sawano(2016)]{saitoh2016theory}
S.~Saitoh and Y.~Sawano.
\newblock \emph{Theory of {Reproducing} {Kernels} and {Applications}},
  volume~44 of \emph{Developments in Mathematics}.
\newblock Springer, Singapore, 2016.
\newblock URL \url{https://doi.org/10.1007/978-981-10-0530-5}.

\bibitem[Sazonov(1958)]{sazonov1958}
V.~V. Sazonov.
\newblock On characteristic functionals.
\newblock \emph{Teor. Veroyatnost. i Primenen.}, 3:\penalty0 201--205, 1958.

\bibitem[Schuster et~al.(2020)Schuster, Mollenhauer, Klus, and
  Muandet]{schuster2019kcdo}
I.~Schuster, M.~Mollenhauer, S.~Klus, and K.~Muandet.
\newblock Kernel conditional density operators.
\newblock In \emph{Proceedings of the 23rd International Conference on
  Artificial Intelligence and Statistics, {AISTATS} 2020, 3--5 June 2020,
  Palermo, Sicily, Italy}, Proceedings of Machine Learning Research, 2020.
\newblock URL \url{https://arxiv.org/abs/1905.11255}.

\bibitem[Smola et~al.(2007)Smola, Gretton, Song, and
  Sch\"olkopf]{smola2007embedding}
A.~J. Smola, A.~Gretton, L.~Song, and B.~Sch\"olkopf.
\newblock A {Hilbert} space embedding for distributions.
\newblock In \emph{Proceedings of the 18th International Conference on
  Algorithmic Learning Theory}, pages 13--31, Berlin, Heidelberg, 2007.
  Springer.
\newblock URL \url{https://doi.org/10.1007/978-3-540-75225-7\_5}.

\bibitem[Song et~al.(2009)Song, Huang, Smola, and Fukumizu]{song2009hilbert}
L.~Song, J.~Huang, A.~Smola, and K.~Fukumizu.
\newblock Hilbert space embeddings of conditional distributions with
  applications to dynamical systems.
\newblock In \emph{Proceedings of the 26th Annual International Conference on
  Machine Learning}, pages 961--968, 2009.
\newblock URL \url{https://doi.org/10.1145/1553374.1553497}.

\bibitem[Song et~al.(2010{\natexlab{a}})Song, Boots, Siddiqi, Gordon, and
  Smola]{song2010a}
L.~Song, B.~Boots, S.~M. Siddiqi, G.~Gordon, and A.~Smola.
\newblock Hilbert space embeddings of hidden {Markov} models.
\newblock In \emph{Proceedings of the 27th International Conference on Machine
  Learning}, ICML2010, pages 991--998, 2010{\natexlab{a}}.
\newblock URL \url{https://dl.acm.org/citation.cfm?id=3104322.3104448}.

\bibitem[Song et~al.(2010{\natexlab{b}})Song, Gretton, and Guestrin]{song2010b}
L.~Song, A.~Gretton, and C.~Guestrin.
\newblock Nonparametric tree graphical models.
\newblock In Y.~W. Teh and M.~Titterington, editors, \emph{Proceedings of the
  Thirteenth International Conference on Artificial Intelligence and
  Statistics}, volume~9 of \emph{Proceedings of Machine Learning Research},
  pages 765--772, 2010{\natexlab{b}}.
\newblock URL \url{http://proceedings.mlr.press/v9/song10a/song10a.pdf}.

\bibitem[Sriperumbudur et~al.(2010)Sriperumbudur, Gretton, Fukumizu,
  Sch\"olkopf, and Lanckriet]{sriperumbudur2010hilbert}
B.~K. Sriperumbudur, A.~Gretton, K.~Fukumizu, B.~Sch\"olkopf, and G.~R.~G.
  Lanckriet.
\newblock Hilbert space embeddings and metrics on probability measures.
\newblock \emph{J. Mach. Learn. Res.}, 11:\penalty0 1517--1561, 2010.
\newblock URL
  \url{http://www.jmlr.org/papers/volume11/sriperumbudur10a/sriperumbudur10a.pdf}.

\bibitem[Sriperumbudur et~al.(2011)Sriperumbudur, Fukumizu, and
  Lanckriet]{sriperumbudur2011universality}
B.~K. Sriperumbudur, K.~Fukumizu, and G.~R.~G. Lanckriet.
\newblock Universality, characteristic kernels and {RKHS} embedding of
  measures.
\newblock \emph{J. Mach. Learn. Res.}, 12:\penalty0 2389--2410, 2011.
\newblock URL
  \url{http://www.jmlr.org/papers/volume12/sriperumbudur11a/sriperumbudur11a.pdf}.

\bibitem[Steinwart and Christmann(2008)]{steinwart2008support}
I.~Steinwart and A.~Christmann.
\newblock \emph{Support {Vector} {Machines}}.
\newblock Information Science and Statistics. Springer, New York, 2008.
\newblock URL \url{https://doi.org/10.1007/978-0-387-77242-4}.

\bibitem[Steinwart et~al.(2006)Steinwart, Hush, and Scovel]{steinwart2006}
I.~Steinwart, D.~Hush, and C.~Scovel.
\newblock An explicit description of the reproducing kernel {Hilbert} spaces of
  {Gaussian} {RBF} kernels.
\newblock \emph{IEEE Trans. Inf. Thy.}, 52\penalty0 (10):\penalty0 4635--4643,
  2006.
\newblock URL \url{https://doi.org/10.1109/tit.2006.881713}.

\bibitem[Weidmann(1980)]{weidmann1980}
J.~Weidmann.
\newblock \emph{Linear {Operators} in {Hilbert} {Spaces}}, volume~68 of
  \emph{Graduate Texts in Mathematics}.
\newblock Springer-Verlag, New York-Berlin, 1980.
\newblock URL \url{https://doi.org/10.1007/978-1-4612-6027-1}.
\newblock Transl.\ J.\ Sz\"{u}cs.

\bibitem[Yuan(2010)]{yuan2010high}
M.~Yuan.
\newblock High dimensional inverse covariance matrix estimation via linear
  programming.
\newblock \emph{J. Mach. Learn. Res.}, 11:\penalty0 2261--2286, 2010.
\newblock URL \url{http://www.jmlr.org/papers/volume11/yuan10b/yuan10b.pdf}.

\end{thebibliography}
